\documentclass[journal]{IEEEtran}
\usepackage[T1]{fontenc}
\usepackage{graphicx}
\usepackage{amssymb}
\usepackage{amsmath}
\usepackage{dcolumn}
\usepackage{bm}
\usepackage{color}
\usepackage{algorithm}
\usepackage{algorithmicx} 
\usepackage{algpseudocode}
\usepackage{url}
\usepackage{nomencl}
\usepackage{multirow}
\usepackage{hyperref} 
\usepackage{fancyhdr}
\usepackage{datetime}
\usepackage{subfigure}
\usepackage{graphicx}  
\usepackage{epstopdf}
\DeclareMathOperator*{\argmin}{argmin}
\DeclareMathOperator*{\argmax}{argmax}

\newcommand{\prior}{\varphi}
\newcommand{\bF}{\mathbf{f}}
\newcommand{\E}{\mathbb{E}}
\newcommand{\noise}{\lambda}
\newcommand{\noiseinv}{\frac{1}{\noise}}
\newcommand{\hyper}{\gamma}

\newcommand{\R}{\mathbb{R}}

\newcommand{\0}{\mathbf{0}}

\newcommand{\bu}{\mathbf{u}}
\newcommand{\bv}{\mathbf{v}}
\newcommand{\bw}{\mathbf{w}}
\newcommand{\bx}{\mathbf{x}}

\newcommand{\by}{\mathbf{y}}

\newcommand{\bA}{\mathcal{A}}

\newcommand{\bI}{\mathbf{I}}
\newcommand{\bL}{\mathcal{L}}

\newcommand{\Prob}{\mathcal{P}}

\newcommand{\bN}{\mathcal{N}}

\newcommand{\bgamma}{\boldsymbol{\gamma}}

\newcommand{\bGamma}{\boldsymbol{\Gamma}}
\newcommand{\bxi}{\boldsymbol{\xi}}
\newcommand{\bXi}{\boldsymbol{\xi}}

\newcommand{\dic}{\mathbf{\Phi}}

\newcommand{\dicc}{\boldsymbol{\Psi}}

\newcommand{\sbl}{\noise\mathbf{I}+\dic\bGamma\dic^\top}
\newcommand{\define}{\triangleq}
\newcommand{\diag}{\operatorname{diag}}

\newcommand{\mean}{\mathbf{m}_{\bw}}
\newcommand{\variance}{\mathbf{\Sigma}_{\bw}}

\newtheorem{proposition}{Proposition}
\newtheorem{assumption}{Assumption}
\newtheorem{remark}{Remark}
\newtheorem{theorem}{Theorem}

\newtheorem{lemma}{Lemma}

\newtheorem{example}{Example}
\newtheorem{proof}{Proof}

\newcommand{\sbsb}{\text{SB}}

\begin{document}

\title{A Sparse Bayesian Approach to the Identification of Nonlinear State-Space Systems
\author{Wei Pan, Ye Yuan, Jorge Gon\c{c}alves and Guy-Bart Stan
	\thanks{W.~Pan and G.-B.~Stan are with the Centre for Synthetic Biology and
		Innovation and the Department of Bioengineering, Imperial College London, United Kingdom. Email: {\tt\small  \{w.pan11,g.stan\}@imperial.ac.uk.}
	Y.~Yuan and J. Gon\c{c}alves are with the Control Group, Department of Engineering, University of Cambridge, United Kingdom. J. Gon\c{c}alves is also with the Luxembourg Centre for Systems Biomedicine, Luxembourg. Email: {\tt\small  \{yy311,jmg77\}@cam.ac.uk.}
	Corresponding Author: Y.~Yuan and G.-B.~Stan.}
}
}
\maketitle

\begin{abstract}
This technical note considers the identification of nonlinear discrete-time systems with additive process noise but without measurement noise. In particular, we propose a method and its associated algorithm to identify the system nonlinear functional forms and their associated parameters from a limited number of time-series data points. For this, we cast this identification problem as a sparse linear regression problem and take a Bayesian viewpoint to solve it. As such, this approach typically leads to nonconvex optimisations. We propose a convexification procedure relying on an efficient iterative re-weighted $\ell_1$-minimisation algorithm that uses general sparsity inducing priors on the parameters of the system and marginal likelihood maximisation. Using this approach, we also show how convex constraints on the parameters can be easily added to the proposed iterative re-weighted $\ell_1$-minimisation algorithm. In the supplementary material \cite{tacappendix}, we illustrate the effectiveness of the proposed identification method on two classical systems in biology and physics, namely, a genetic repressilator network and a large scale network of interconnected Kuramoto oscillators.
\end{abstract}

\section{INTRODUCTION}\label{sec:introduction}

Identification from time-series data of nonlinear discrete-time state-space systems with additive process noise is relevant to many different fields such as systems/synthetic biology, econometrics, finance, chemical engineering, social networks, etc. Yet, the development of general identification techniques remains challenging, especially due to the difficulty of adequately identifying nonlinear systems~\cite{ljung1999system,billings2013nonlinear}. Nonlinear dynamical system identification aims at recovering the set of nonlinear equations associated with the system from time-series observations. The importance of nonlinear dynamical system identification and its associated difficulties have been widely recognised \cite{billings2013nonlinear,sjoberg1995nonlinear}.

Since, typically, nonlinear functional forms can be expanded as sums of terms belonging to a family of parameterised functions (see \cite[Sec.~5.4]{ljung1999system} and \cite{billings2013nonlinear}), an usual approach to identify nonlinear state-space models is to search amongst a set of possible nonlinear terms (e.g., basis functions) for a parsimonious description coherent with the available data \cite{haber1990structure}. A few choices of basis functions are provided by classical functional decomposition methods such as Volterra expansion, Taylor polynomial expansion or Fourier series~\cite{ljung1999system,billings2013nonlinear,barahona1996detection}. This is typically used to model systems such as those described by Wiener and Volterra series \cite{barahona1996detection,wiener1966nonlinear}, neural networks \cite{narendra1990identification}, nonlinear auto-regressive with exogenous inputs (NARX) models \cite{leontaritis1985input}, and Hammerstein-Wiener \cite{bai1998optimal} structures, to name just a few examples.

Recently, graphical models have been proposed to capture the structure of nonlinear dynamical networks. In the standard graphical models where each state variable represents a node in the graph and is treated as a random variable, the nonlinear relations among nodes can be characterised by factorising the joint probability distribution according to a certain directed graph \cite{kollar2009probabilistic,pearl1988probabilistic,spirtes2000causation}. However, standard graphical models are often not adequate for dealing with times series directly. This is mainly due to two aspects inherent to the construction of graphical models. The first aspect pertains to the efficiency of graphical models built using time series data. In this case, the building of graphical models requires the estimation of conditional distributions with a large number of random variables \cite{barber2010graphical} (each time series is modelled as a finite sequence of random variables), which is typically not efficient. The second aspect pertains to the estimation of the moments of conditional distribution, which is very hard to do with a limited amount of data especially when the system to reconstruct is nonlinear. In the case of linear dynamical systems, the first two moments can sometimes be estimated from limited amount of data \cite{bach2004learning,materassi2012problem}. However, higher moments typically need to be estimated if the system under consideration is nonlinear.

In this technical note, we propose a method to alleviate the problems mentioned above. This method relies on the assumption that there exits a finite set of candidate dictionary functions whose linear combination allows to describe the dynamics of the system of interest. In particular, we focus on discrete-time nonlinear systems with additive noise represented in a general state-space form. Based on this, we develop an identification framework that uses time series data and \emph{a priori} knowledge of the type of system from which these time series data have been collected, e.g.,\ biological, biochemical, mechanical or electrical systems. For example in Genetic Regulatory Network (GRN), only polynomial or rational nonlinear functional forms typically need to be considered in the identification process.

To identify the network efficiently given the available time series data, we cast this nonlinear system identification problem as a sparse linear regression problem \cite{pan2012cdc,Candes2005decoding,donoho2006compressed}. Although such problems have been widely applied in the context of sparse coding, dictionary learning or image processing \cite{mairal2008discriminative,mairal2010online}, they have received little attention in nonlinear dynamical system identification. Besides the work presented here, one of the rare example of sparse estimation technique used for dynamical system identification is the multiple kernel-based regularisation method, which has been used to estimate finite impulse response models \cite{TianshiTAC}.

Furthermore, very few contributions are available in the literature that address the identification problem with \emph{a priori} information or constraints on the parameters of the system \cite{cerone2011enforcing,zavlanos2011inferring}. In contrast, our proposed framework allows us to incorporate convex constraints on the associated model parameters, e.g.,\ equality or inequality constraints imposed among parameters, or \emph{a priori} required stability conditions.

In sparse linear regression problems, finding the sparsest solution is desirable but typically NP-hard. The classic ``Lasso'' or $\ell_{1}$-minimisation algorithm are typically used as a relaxation to alleviate this numerical difficulty \cite{tibshirani1996regression}. However, these algorithms usually only work well or have performance guarantees when the considered dictionary matrix has certain properties such as the \textit{restricted isometry property} (RIP) \cite{Candes2005decoding,dai2009subspace} or the \emph{incoherence} property \cite{donoho2003optimally}. Loosely speaking, these properties require that the columns of the dictionary matrix are orthogonal, or nearly so. Unfortunately, such properties are hardly guaranteed for nonlinear identification problems and, as a consequence, $\ell_{1}$-relaxation based algorithms typically do not work well when these conditions are not satisfied.

In this technical note, we shall explain, from a probabilistic viewpoint, how a Bayesian approach can attenuate problems arising in the case of high correlations between columns of the dictionary matrix. In particular, the main contributions of this technical note are:
\begin{itemize}
\item To formulate the problem of reconstructing discrete-time nonlinear systems with additive noise into a sparse linear regression problem. The model class in this technical note covers a large range of systems, e.g.,\ systems with multiple inputs and multiple outputs, systems with memory in their states and inputs, and autoregressive models.
\item To derive a sparse Bayesian formulation of the nonlinear system identification problem, which is casted into a nonconvex optimisation problem.
\item To develop an iterative re-weighted $\ell_1$-minimisation algorithm to convexify the nonconvex optimisation problem and solve it efficiently.
This formulation can also take into account additional convex constraints on the parameters of the model.
\end{itemize}

The generality of our framework allows it to be applied on a broad class of nonlinear system identification problems. In particular, to illustrate our results, we applied our approach to two examples: (1) the Genetic Repressilator Network, where we identify nonlinear regulation relationships between genes, transcriptional and translational strengths and degradation rates, and (2) a network of Kuramoto Oscillators, where we identify the network topology and nonlinear coupling functions. Details about these examples can be found in the supplementary material \cite{tacappendix}.

This technical note is organised as follows. Section \ref{sec:model} introduces the class of nonlinear models considered. Section \ref{sec:probstate} formulates the nonlinear identification problem into a sparse linear regression problem. Section~\ref{sec:bayesian} re-interprets the sparse problem from a Bayesian point of view, while Section~\ref{sec:cost1} shows how the resulting nonconvex optimisation problem can be convexified and solved efficiently using an iterative re-weighted $\ell_1$-minimisation algorithm. Finally, we conclude and discuss several future open questions.

\section{FORMULATION OF THE NONLINEAR IDENTIFICATION PROBLEM}\label{sec:identification}

\subsection{Considered Nonlinear Dynamical Model Class}\label{sec:model}

We consider dynamical systems described by discrete-time nonlinear state-space equations driven by additive Gaussian noise. The discrete-time dynamics of the $i$-th state variable $x_{i}$, $i=1,\ldots,n_{\bx}$ is assumed to be described by:
\begin{equation}
\begin{aligned}
x_i(t_{k+1})
& = \mathbf{F}_i(\bx(t_{k}),\bu(t_{k}))+\xi_i(t_{k})  \\
&=\sum\nolimits_{s=1}^{N_{i}}v_{is}f_{is}(\bx(t_{k}),\bu(t_{k}))+\xi_i(t_{k})  \\
&=\bF^\top_i(\bx(t_{k}),\bu(t_{k}))\bv_i+\xi_i(t_{k}),
\label{eq:expansion}
\end{aligned}
\end{equation}
where $\bx=[x_{1},\ldots ,x_{n_\bx}]^\top\in {\R}^{n_\bx}$ denotes the state vector, $\bu=[u_{1},\ldots ,u_{n_\bu}]^\top\in {\R}^{n_\bu}$ denotes the input vector, and $\mathbf{F}_i(\cdot): \mathbb{R}^{n_\bx+n_\bu}\rightarrow \mathbb{R}$ is a smooth nonlinear function which is assumed to be represented as a linear combination of several dictionary functions $f_{is}(\bx(t_{k}),\bu(t_{k})): \mathbb{R}^{n_\bx+n_\bu}\rightarrow \mathbb{R}$ (see Sec.~5.4 in \cite{ljung1999system}). These constituent dictionary functions can be monomial, polynomial, constant or any other functional form such as rational, exponential, trigonometric etc. $\bF_i(\bx(t_{k}),\bu(t_{k}))$ is the vector of considered dictionary functions (which does not contain unknown parameters) while $\bv_i \in {\mathbb{R}^{N_i}}$ appearing in \eqref{eq:expansion} is the weight vector associated with the dictionary functions vector. The additive noise $\xi_i(t_{k})$ is assumed to be i.i.d. Gaussian distributed with zero mean: $\xi_i(t_{k})\thicksim\bN(0, \noise_i)$, with
$
\E(\xi_i (t_{p}))=0, \ \E(\xi_i(t_{p})\xi_i(t_{q}))=\noise_i\delta_{pq},
\label{model:noise}
$
where
$
\delta _{pq}=\left\{
\begin{array}{ll}
1, & p=q, \\
0, & p\neq q
\end{array}
\right.
$.
$\xi_i(\cdot)$ and $\xi_j(\cdot)$ are assumed independent $\forall i \neq j$.

\begin{remark}
The class of systems considered in (\ref{eq:expansion}) can be extended to the more general dynamics class
$
\bx_i(t_{k+1})=\mathbf{F}_i(\bx(t_{k}),\ldots, \bx(t_{k-m_\bx}), \bu(t_{k}),\ldots,\bu(t_{k-m_\bu}) )+\bxi(t_{k}),
$
where the ``orders'' $m_\bx$ and $m_\bu$ are assumed to be known \emph{a priori}, and $\mathbf{F}_i(\cdot): \mathbb{R}^{(m_{\bx}+1)n_\bx+(m_{\bu}+1)n_\bu}\rightarrow \mathbb{R}$. An example of such system can be found in the supplementary material \cite{tacappendix} (see Example 1). In particular, MIMO nonlinear autoregressive models belong to such descriptions.
\end{remark}

\subsection{Identification Problem Statement}
\label{sec:probstate}

If $M$ data samples satisfying (\ref{eq:expansion}) can be obtained from the system of interest, the system in (\ref{eq:expansion}) can be written as
$
\mathbf{y}_i=\dicc_i\bv_i+\bXi_i, \ i=1,\ldots,n_{\bx},
$
where
$\by_i \define\left[x_i(t_{1}),\ldots,x_i(t_{M})\right]^\top\in {\mathbb{R}}^{M\times 1}$,
$\bv_i \define \left[v_{i1},\ldots,v_{iN_i}\right]^\top \in {\mathbb{R}}^{N_i\times 1}$,
$\bXi_i \define
\left[\xi_i(t_{0}),\ldots,\xi_i(t_{M-1})\right]^\top\in {\mathbb{R}}^{M\times 1}$, and $\mathbf{\dicc}_i\in {\mathbb{R}}^{M\times N_i}$ represents the dictionary matrix with its $j$-th column being $[f_{ij}(\bx(t_{0}),\bu(t_{0})), \ldots, f_{ij}(\bx(t_{M-1}),\bu(t_{M-1}))]^{\top}$.

In this framework, the identification problem amounts to finding $\bv_i \in \mathbb{R}^{N_i \times 1}$ given the measured data stored in $\mathbf{y}_i$. This, in turn, amounts to solving a linear regression problem, which can be done using standard least square approaches, provided that the structure of the nonlinearities in the model are known, i.e., provided that $\dicc_i$ is known. In what follows, we make the following assumption on the measurements contained in $\by_i$.
\begin{assumption}
\label{assumption:observable}
The system (\ref{eq:expansion}) is fully measurable, i.e., time series data of all the state variables $x_i$ can be obtained.
\end{assumption}

Depending on the field for which the dynamical model needs to be built, only a few typical nonlinearities specific to this field need to be considered. In what follows we gather in a matrix $\dic_i$ similar to $\dicc_i$ the set of \emph{all} candidate/possible  dictionary functions that we want to consider for identification:
\begin{equation}
\mathbf{y}_i=\dic_i\bw_i+\bXi_i, \ i=1,\ldots,n_{\bx}.
\label{problem:expand}
\end{equation}
The solution $\bw_i$ to \eqref{problem:expand} is typically going to be sparse,  which is mainly due to the potential introduction of non-relevant and/or non-independent dictionary functions in $\dic_i$.

Since the $n_{\bx}$ linear regression problems in (\ref{problem:expand}) are independent, for simplicity of notation, we omit the subscript $i$ used to index the state variable and simply write:
\begin{equation}
\mathbf{y}=\dic \bw+\bXi.
\label{problem}
\end{equation}
It should be noted that $N$, the number of dictionary functions or number of columns of the dictionary matrix $\dic \in \mathbb{R}^{M \times N}$, can be very large, at least larger than the number of observations $M$. Moreover, since $\mathbf{y}$ is constructed from time series data, typically two or more of the columns of the $\dic$ matrix are highly correlated. In this case standard methods, which involve some form of $\ell_1$-regularised minimisation, often yield poor performance on system identification~\cite{Candes2006robust}. 

\section{BAYESIAN VIEWPOINT ON THE RECONSTRUCTION PROBLEM}
\label{sec:bayesian}

\subsection{Sparsity Inducing Priors }
Bayesian modelling treats all unknowns as stochastic variables with certain probability distributions \cite{bishop2006pattern}. For $\by=\dic \bw+\bXi$, it is assumed that the stochastic variables in the vector $\bXi$ are Gaussian i.i.d. with $\bXi\thicksim\bN(\mathbf{0}, \noise\bI)$.
In such case, the likelihood of the data given $\bw$ is
$
\Prob(\by|\bw)
=\mathcal{N}(\by|{\dic} {\bw},\noise\bI)
\propto\exp \left[ -\frac{1}{2 \noise}\| \by-\dic\bw\|_2^{2}\right].
$
We define a prior distribution $\Prob(\bw)$ as
$
\Prob(\bw)\propto\exp \left[-\frac{1}{2}\sum_{j}g(w_j)\right]=\prod_{j}\exp \left[-\frac{1}{2}g(w_j)\right]=\prod_{j}\Prob(w_j),
$
where $g(w_j)$ is a given function of $w_j$.
To enforce sparsity on $\bw$, the function $g(\cdot)$ is usually chosen as a concave, non-decreasing function of $|w_j|$. Examples of such functions $g(\cdot)$ include Generalised Gaussian priors and Student's \emph{t} priors (see \cite{palmer2006variational} for details).

Computing the posterior mean $\E(\bw|\by)$ is typically intractable because the posterior $\Prob(\bw|\by)$ is highly coupled and non-Gaussian.
To alleviate this problem, ideally one would like to approximate $\Prob(\bw|\by)$ as a Gaussian distribution for which efficient algorithms to compute the posterior exist \cite{bishop2006pattern}. Another approach consists in considering \emph{super-Gaussian} priors, which yield a lower bound for the priors $\Prob(w_j)$ \cite{palmer2006variational}.
The sparsity inducing priors mentioned above are \emph{super-Gaussian}. More specifically, if we define $\boldsymbol{\hyper} \define \left[\hyper_1, \ldots, \hyper_N\right]^\top \in \R^N_{+}$, we can represent the priors in the following relaxed (variational) form:
$
\Prob(\bw)=\prod_{j=1}^{n}\Prob(w_j)$, $\Prob(w_j)=\max_{\hyper_j>0}\bN(w_j|0,\hyper_j)\prior(\hyper_j)
$,
where $\prior(\hyper_j)$ is a nonnegative function which is treated as a hyperprior with $\hyper_j$ being its associated hyperparameters. Throughout, we call $\prior(\hyper_j)$ the ``\emph{potential function}''. This Gaussian relaxation is possible if and only if $\log \Prob(\sqrt{w_j})$ is concave on $(0,\infty)$.
The following proposition provides a justification for the above:

\begin{proposition}\cite{palmer2006variational}
\label{super-Gaussian}
A probability density $\Prob(w_j)\equiv \exp(-g(w_j^2))$ can be represented in the convex variational form:
$
\Prob(w_j)=\max_{\hyper_j>0}\bN(w_j|0,\hyper_j)\prior(\hyper_j)
$
if and only if $-\log \Prob(\sqrt{w_j})=g(w_j)$ is concave on $(0,\infty)$. In this case  the potential function takes the following expression:
$
\prior(\hyper_j)=\sqrt{{2\pi}/{\hyper_j}}\exp\left(g^{*}\left({\hyper_j}/{2}\right)\right)
$
where $g^{*}(\cdot)$ is the concave conjugate of $g(\cdot)$.
A symmetric probability density $\Prob(w_j)$ is said to be super-Gaussian if  $\Prob(\sqrt{w_j})$ is log-convex on $(0,\infty)$.
\end{proposition}

\subsection{Marginal Likelihood Maximisation}
For a fixed $\bgamma=\left[{\hyper}_{1},\ldots, {\hyper}_{N}\right]$,
we define a relaxed prior, which is a joint probability distribution over $\bw$ and ${\bgamma}$, as
$
\Prob(\bw; \bgamma)
=\prod_{j}\bN(w_j|0,{\hyper}_j)\prior({\hyper}_j)
=\Prob(\bw|\bgamma)\Prob(\bgamma) \le \Prob(\bw)
$,
where $\Prob(\bw|\bgamma)\define \prod_{j}\bN(w_j|0,{\hyper}_j),
\Prob(\bgamma) \define \prod_{j} \prior({\hyper}_j)$.
Since the likelihood is  $\Prob(\by|\bw)$ is Gaussian, we can get a relaxed posterior which is also Gaussian
$
\Prob(\bw|\by,\bgamma)
=\frac{\Prob(\by|\bw)\Prob(\bw;\bgamma)}{\int \Prob(\by|\bw)\Prob(\bw;\bgamma)d\bw} 
=\bN(\mean,\variance).
$
Defining $\bGamma\define \diag[\bgamma]$, the posterior mean and covariance are given by:
\begin{eqnarray}
\mean&=& \bGamma \dic^\top (\sbl)^{-1} \by, \label{mean} \\
\variance&=& \bGamma- \bGamma \dic^\top (\sbl)^{-1} \dic. \label{variance}
\end{eqnarray}
Now the key question is how to choose the most appropriate $\bgamma=\hat{\bgamma}=\left[\hat{\hyper}_{1},\ldots, \hat{\hyper}_{N}\right]$ to maximise $\prod_{j}\bN(w_j|0,{\hyper}_j)\prior({\hyper}_j)$ such that $\Prob(\bw|\by,\hat{\bgamma})$ can be a ``good'' relaxation to $\Prob(\bw|\by)$.
Using the product rule for probabilities, we can write the full posterior as:
$
\Prob(\bw, \bgamma|\by)
\propto \Prob(\bw|\by,\bgamma)\Prob(\bgamma|\by)
= \bN(\mathbf{m}_{\bw},\mathbf{\Sigma}_{\bw}) \times {\Prob(\by|\bgamma)\Prob(\bgamma)}/{\Prob(\by)}.
$
Since $\Prob(\by)$ is independent of $\bgamma$, the quantity $\Prob(\by|\bgamma)\Prob(\bgamma)=\int \Prob(\by|\bw)\Prob(\bw|\bgamma)\Prob(\bgamma)d\bw$ is the prime target for variational methods \cite{wainwright2008graphical}. This quantity is  known as evidence or marginal likelihood. A good way of selecting $\hat{\bgamma}$ is to choose it as the minimiser of the sum of the misaligned probability mass, e.g.,
\begin{equation}
\begin{aligned}
\hat{\bgamma}&=  \argmin\limits_{\bgamma\geq\mathbf{0}} \int \Prob(\by|\bw)\left|\Prob(\bw)-\Prob(\bw;\bgamma)\right|d\bw \\
&= \argmax\limits_{\bgamma\geq\mathbf{0}} \int \Prob(\by|\bw)\prod_{j=1}^{n}\bN(w_j|0,\hyper_j)\prior(\hyper_j)d\bw.
\label{mass}
\end{aligned}
\end{equation}
The second equality is  a consequence of $\Prob(\bw;\bgamma)\le\Prob(\bw)$. The procedure in (\ref{mass}) is referred to as  evidence maximisation or type-II maximum likelihood \cite{tipping2001sparse}. It means that the marginal likelihood can be maximised by selecting the most probable hyperparameters able to explain the observed data. Once $\hat{\bgamma}$ is computed, an estimate of the unknown weights can be obtained by setting $\hat{\bw}$ to the posterior mean (\ref{mean}) as $\hat{\bw}=\E(\bw|\by;\hat{\bgamma})=\hat{\bGamma} \dic^\top (\noise\mathbf{I+{\dic}\hat{\bGamma}}{{\dic}}^\top)^{-1} \by$, with $\hat{\bGamma}\define \diag[\hat{\bgamma}]$. If an algorithm can be proposed to compute $\hat{\bgamma}$ in (\ref{mass}), we can, based on it, obtain an estimation of the posterior mean $\hat{\bw}$.

\subsection{Enforcing Additional Constraints on $\bw$}

It is often important to be able to impose constraints on $\hat{\bw}$ when formulating the optimisation problem (\ref{mass}) used to compute $\hat{\bw}$ from $\hat{\bgamma}$. In physical and biological systems, positivity of the parameters $\bw$ of the system is an example of such constraints. Another example of constrained optimisation comes from stability considerations, which emerge naturally when the underlying system is known \emph{a priori} to be stable\footnote{Many stability conditions can be formulated as convex optimisation problems (see for example \cite{boyd1987linear,horn1990matrix}).}. Yet, only a few contributions in the literature address the problem of how to take into account \emph{a priori} information on system stability in the context of system identification \cite{cerone2011enforcing,zavlanos2011inferring}. To be able to integrate constraints on $\bw$ into the problem formulation, we consider the following assumption on $\bw$.
\begin{assumption}
\label{assumption-constraints}
Constraints on the weights $\bw$ can be described by a set of convex functions:
\begin{equation}
\begin{aligned}
H^{[I]}_{i}(\bw)&\leq0,\ \ i=1,\ldots, m_I, \\
H^{[E]}_{j}(\bw)&=0,\ \ j=1,\ldots, m_E.
\label{convexconstraints}
\end{aligned}
\end{equation}
where the convex functions $H^{[I]}_{i}: \R^{N}\rightarrow \R$ are used to define inequality constraints, whereas the convex functions $H^{[E]}_{j}: \R^{N}\rightarrow \R$ are used to define equality constraints.
\end{assumption}

\section{NONCONVEX OPTIMISATION FOR IDENTIFICATION PROBLEMS}\label{sec:cost1}

In this section, we derive a sparse Bayesian formulation of the problem of system identification with convex constraints, which is casted into a nonconvex optimisation problem. The nonconvex optimisation problem can be dealt by an iterative re-weighted $\ell_1$-minimisation algorithm.

\subsection{Nonconvex Objective Function in Hyperparameter}
\label{sec:hyper}

\begin{theorem}
\label{theorem:hyper}
The optimal hyperparameters $\hat{\bgamma}$ in (\ref{mass}) can be obtained by minimising the following objective function
\begin{equation}
\begin{aligned}
{\bL_{\bgamma}}\left( \bgamma \right) &= \log \left\vert \sbl \right\vert \\
&+\by^\top(\sbl)^{-1}\by+\sum\nolimits_{j=1}^{N}p(\hyper_j),
\label{hypercostfunction}
\end{aligned}
\end{equation}
where $p(\hyper_j)=-2\log \prior(\hyper_j)$.
The posterior mean is then given by
$\hat{\bw}=\hat{\bGamma} \dic^\top (\noise\mathbf{I+{\dic}\hat{\bGamma}}{{\dic}}^\top)^{-1} \by$,
where $\hat{\bGamma}=\diag[\hat{\bgamma}]$.

\end{theorem}
\begin{proof}
See Section A in the Appendix \cite{tacappendix}.
\end{proof}

\begin{lemma}\label{lemma:concave}
The objective function in the hyperparameter $\bgamma$-space, ${\bL_{\bgamma}}\left( \bgamma \right)$ in (\ref{hypercostfunction}), is nonconvex.
\end{lemma}

\begin{proof}
See Section B in the Appendix \cite{tacappendix}.
\end{proof}

\subsection{Nonconvex Objective Function in $\bw$ with Convex Constraints }
\label{sec:with-constraint}

Based on the analysis in Section \ref{sec:hyper}, we first derive a dual objective function in the $\bw$-space with convex constraints by considering the equivalent objective function of (\ref{hypercostfunction}) in the $\bgamma$-space. We then show that this equivalent objective function is also nonconvex.

\begin{theorem}
\label{theorem:MAP}
The estimate for $\bw$ with constraints can be obtained by solving the optimisation problem
\begin{equation}
\label{w-nonconvex}
\begin{split}
\min_{\bw} \|\by-\dic\bw\|_{2}^{2}+\noise g_{\sbsb}(\bw), \,\,\,\,\, \mathrm{subject}\,\,\mathrm{to}  \,\,\,\,\
(\ref{convexconstraints})
\end{split}
\end{equation}
where
$
g_{\sbsb}(\bw)=\min\limits_{\bgamma\geq \mathbf{0}}\{\bw^\top\bGamma^{-1}\bw+\log|\sbl|+\sum\nolimits_{j=1}^{N}p(\hyper_j)\} \label{penalty}
$
and the estimate of the stochastic variable $\bw$ is given by the poseterior mean $\mean$ defined in (\ref{mean}).
\end{theorem}

\begin{proof}
See Section C in the Appendix \cite{tacappendix}.
\end{proof}

Although all the constraint functions are convex in Theorem \ref{theorem:MAP}, we show in the following Lemma that the objective function in \eqref{w-nonconvex} is nonconvex since it is the sum of convex and concave functions.
\begin{lemma}
\label{lemma:gsb}
The penalty function $g_{\sbsb}(\bw)$ in Theorem \ref{theorem:MAP} is a non-decreasing, concave function of $|\bw|$ which promotes sparsity on the weights $\bw$.
\end{lemma}
\begin{proof}
The proof uses the duality lemma (see Sec. 4.2 in \cite{jordan1999introduction}).
See Section D in the Appendix \cite{tacappendix}.
\end{proof}

\subsection{Lasso Type Algorithm}
\label{sec:cost-constraint-convex}

We define the terms excluding $h^{*}(\bgamma^{*})$ as
\begin{equation}
\begin{aligned}
\bL_{\bgamma^{*}}(\bgamma,\bw)\define   \noiseinv\| \by-\dic\bw\|_{2}^{2}+\sum\nolimits_j\left(\frac{w_j^2}{\hyper_j}+\hyper_j^{*}\hyper_j\right).
\end{aligned}
\end{equation}
For a fixed $\bgamma^{*}$, we notice that $\bL_{\bgamma^{*}}(\bgamma,\bw)$  is jointly convex in $\bw$ and $\bgamma$ and can be globally minimised by solving over $\bgamma$ and then $\bw$.
Since ${w_j^2}/{\hyper_j}+\hyper_j^{*}\hyper_j \geq 2w_j\sqrt{\bgamma^{*}_{j}}$, for any $\bw$,
$\gamma_j={|{w}_j|}/{\sqrt{\gamma^{*}_{j}}}$ minimises $\bL_{\bgamma^{*}}(\bgamma,\bw)$.
When  $\gamma_j={|{w}_j|}/{\sqrt{\gamma^{*}_{j}}}$ is substituted into $\bL_{\bgamma^{*}}(\bgamma,\bw)$, $\hat{\bw}$ can be obtained by solving the following weighted convex $\ell_1$-minimisation procedure
\begin{equation}
\begin{aligned}
\mathbf{\hat{w}}
=\argmin\limits_{\bw}\left\{\| \by-\dic
\bw\|_{2}^{2}+2\noise \sum\nolimits_{j=1}^{N}\sqrt{\bgamma_{j}^{*}}|w_j| \right\}.
\label{wupdate}
\end{aligned}
\end{equation}
We can then set
$
\gamma_j=|\hat{w}_j|/\sqrt{\gamma^{*}_{j}}
\label{gammaupdate}
$, $\forall j$.
As a consequence, $\bL_{\bgamma^{*}}(\bgamma,\bw)$ will be minimised for  any fixed $\bgamma^{*}$.
Due to the concavity of $g_{\sbsb}(\bw)$, the objective function in (\ref{w-nonconvex}) can be optimised using a re-weighted $\ell_1$-minimisation in a similar way as was considered in (\ref{wupdate}). The updated weight at the $k^{th}$ iteration is then given by
$
u_j^{(k)}\define \left.\frac{\partial g_{\sbsb}(\bw)}{2\partial |w_j|}\right|_{\bw=\bw^{(k)}}=\sqrt{\bgamma^{*}_j}.
$

We can now explain how the update of the parameters can be performed based on the above. We start by setting the iteration count $k$ to zero and $u_j^{(0)}=1, \ \forall j$. At this stage, the solution is a typical $\ell_1$-minimisation solution. Then at the $k^{th}$ iteration, we initialise $u_j^{(k)}=\sqrt{\bgamma^{*(k)}_j}$, $\forall j$ and then minimise over $\bgamma$ using $\hyper_j=|w_j|/\sqrt{\hyper^{*}_{j}},\ \forall j$.
Consider again $\bL_{\bgamma,\bw}(\bgamma,\bw)$. For any fixed $\bgamma$ and $\bw$, {the tightest bound} can be obtained by minimising over $\bgamma^{*}$. The tightest value of $\bgamma^{*}=\hat{\bgamma^{*}}$ equals the gradient of the function $h(\bgamma)\define \log|\sbl|+\sum\nolimits_{j=1}^{N}p(\hyper_j)$ defined in Lemma \ref{lemma:concave} at the current $\bgamma$. $\bgamma^{*}$ has the following analytical expression:
$
\hat{\bgamma^{*}}
=\nabla_{\bgamma}\left( \log |\sbl|+\sum\nolimits_{j=1}^{N}p(\hyper_j) \right)
= \diag \left[\dic^\top\left(\sbl\right)^{-1}{\dic} \right]+p'(\bgamma),
$
where $p'(\bgamma)=\left[p'(\hyper_1),\ldots,p'(\hyper_N)\right]^\top$. The optimal $\bgamma^{*(k+1)}$ can then be obtained as
$
\bgamma^{*(k+1)}= \diag \left[\dic^\top\left(\noise\mathbf{I}+\dic \mathbf{\Gamma}^{(k)} \dic^\top\right)^{-1}{\dic} \right]+p'(\bgamma^{(k)}).
$
After computing the estimation of ${\hyper_j}^{(k)}={|w_j^{(k)}|}/{\sqrt{\hyper^{*(k)}_{j}}}$,
we can compute $\bgamma^{*(k+1)}$, which gives
$
\hyper_j^{*(k+1)}= \dic_j^\top\left(\noise\mathbf{I}+\dic \mathbf{U}^{(k)} \mathbf{W}^{(k)} \dic^\top\right)^{-1}{\dic_j}+p^{\prime}(\hyper_j^{(k)}),
$
where
$
\mathbf{{\Gamma}}^{(k)}\define \diag\left[{\bgamma}^{(k)}\right]$, $
\mathbf{U}^{(k)}\define \diag\left[\mathbf{u}^{(k)}\right]^{-1}=\diag\left[\sqrt{\bgamma^{*(k)}}\right]^{-1} $,
$
\mathbf{W}^{(k)}\define \diag\left[|\bw^{(k)}|\right].
$
We can then define
$
u_j^{(k+1)}\define \sqrt{\gamma_j^{*(k+1)}}
$
for the next iteration of the weighted $\ell_1$-minimisation. The above described procedure is summarised in Algorithm \ref{alg:weight}.
\begin{algorithm}[!]
\caption{Nonlinear Identification Algorithm}
\label{alg:weight}
\begin{algorithmic}[1]
	\State Collect time series data from the system of interest (assuming the system can be described by~\eqref{eq:expansion});
	\State Select the candidate dictionary functions that will be used to construct the dictionary matrix described in Section~\ref{sec:probstate};
	\State Initialise $u_j^0=1, \ \forall j$
	\For  {$k=0, \ldots, k_{\max}$}
	\State Solve the weighted $\ell_1$-minimisation problem with convex constraints on $\bw$
	\begin{equation}
	\begin{split}
	\min_{\bw} \| \by-\dic\bw\|_{2}^{2}+2\noise\sum\nolimits_{j}u_j^{(k)}|w_j|, \,\,\,\,\, \mathrm{subject}\,\,\mathrm{to}  \,\,
	(\ref{convexconstraints}); \notag
	\end{split}
	\end{equation}
	\State Set $\mathbf{U}^{(k)}\define \diag\left[\mathbf{u}^{(k)}\right]^{-1}, \
	\mathbf{W}^{(k)}\define \diag\left[|\bw^{(k)}|\right];$
	\State Update weights $u_j^{(k+1)}$ for the next iteration 
	$ u_j^{(k+1)}=\left[\dic_j^\top\left(\noise\mathbf{I}+\dic \mathbf{U}^{(k)} \mathbf{W}^{(k)} \dic^\top\right)^{-1}{\dic_j}+p^{\prime}(\hyper_j^{(k)})\right]^{{1}/{2}}$;
	
	\If   {a stopping criterion is satisfied}
	\State Break;
	\EndIf
	\EndFor
\end{algorithmic}
\end{algorithm}

\begin{remark}
\label{remark:threshold}
There are two important aspects of the re-weighted $\ell_1$-minimisation algorithm presented in Algorithm~\ref{alg:weight}. First, for convex optimisation, there will be no exact zeros during the iterations and strictly speaking, we will always get a solution without any zero entry even when the RIP condition holds. However, some of the estimated weights will have very small magnitudes compared to those of other weights, e.g., $\pm 10^{-5}$ compared to $1$, or the ``energy'' some of the estimated weights will be several orders of magnitude lower than the average ``energy'', e.g., $\|w_j\|_2^2 \ll \|\bw\|_2^2$. Thus a threshold needs to be defined \emph{a priori} to prune ``small'' weights at each iteration. The second aspect concerns the computational complexity of this approach. The repeated execution of Algorithm \ref{alg:weight} is very cheap computationally since it scales as $\mathcal{O}(MN\|\bw^{(k)}\|_0)$ (see \cite{candes2008enhancing, wipf2010iterative}). Since at each iteration certain weights are estimated to be zero, certain dictionary functions spanning the corresponding columns of $\dic$ can be pruned out for the next iteration.
\end{remark}

\subsection{Convergence}

It is natural to investigate the convergence properties of this iterative re-weighted $\ell_1$-minimisation procedure. Let $\bA(\cdot)$ denote a mapping that assigns to every point in $\R^{N}_{+}$ the subset of  $\R^{N}_{+}$ which satisfies Steps 5 and 6 in Algorithm \ref{alg:weight}. Then the convergence property can be established as follows:

\begin{theorem}
\label{theorem:convergence}
Given the initial point $\bgamma^{(0)}\in \R^{n}_{+}$ a sequence $\{\bgamma^{(k)}\}_{k=0}^{\infty}$ is generated such that
$
\bgamma^{(k+1)}\in \bA(\bgamma^{(k)}), \, \forall k.
$
This sequence is guaranteed to converge to a local minimum (or saddle point) of ${\bL_{\bgamma}}$ in (\ref{hypercostfunction}).
\end{theorem}
\begin{proof}
The proof is in one-to-one correspondence with that of the Global Convergence Theorem \cite{zangwill1969nonlinear}.
See Section E in the Appendix \cite{tacappendix}.
\end{proof}

\section{Illustrative Numerical Examples}\label{sec:examples}

To implement Algorithm~\ref{alg:weight}, we use CVX, a popular package for specifying and solving convex programs \cite{grant2008cvx}.
To illustrate our results, the approach is applied to two classic examples: (1) the Genetic Repressilator Network, where we identify nonlinear regulation relationships between genes, transcriptional and translational strengths and degradation rates, and (2) a network of Kuramoto Oscillators, where we identify the network topology and nonlinear coupling functions. More details about these two examples and algorithmic comparisons with other algorithms described in \cite{sparselab} in terms of the Root of the Normalised Mean Square Error (RNMSE) and computational running time for different Signal-to-Noise Ratios (SNR) can be found in the supplementary material \cite{tacappendix}.
Importantly, this comparison shows that Algorithm~\ref{alg:weight} outperforms other classical algorithms \cite{sparselab} in terms of RNMSE, when used to identify the nonlinear systems associated with these illustrative examples.

\section{CONCLUSION AND DISCUSSION}\label{sec:conclusion}
This technical note proposed a new method for the identification of nonlinear discrete-time state-space systems with additive process noise. This method only required time-series data and some prior knowledge about the type of system from which these data have been acquired (e.g.,\ biochemical, mechanical or electrical). Based on this prior knowledge, candidate nonlinear functions (dictionary functions) can be selected for the particular type of system to be identified.

Due to the typical sparsity in terms of number of dictionary functions used to describe the dynamics of nonlinear systems and the fact that the number of measurements is typically small (at least smaller than the number of candidate nonlinear functions), the corresponding identification problem falls into the class of sparse linear regression problems. We considered this problem in a Bayesian framework and solved it efficiently using an iterative re-weighted $\ell_1$-minimisation algorithm. This approach also allowed us to easily add convex constraints from prior knowledge of some properties of the system (e.g.,\ positivity of certain variables, stability of the system, etc.). Finally, we illustrated how this approach can be efficiently used to accurately reconstruct discrete-time nonlinear models of the genetic repressilator and of Kuramoto networks.

Several important questions remain currently open for further research. Possibly, the most important is the assumption that the system is fully measurable. Typically, only part of the state is measured~\cite{yuan2011robust,yuan2012decentralised}, and, in particular, the number of hidden/unobservable nodes and their position in the network are usually unknown. We are currently investigating partial-measurement extensions of the method presented in this technical note.  Meanwhile, our algorithm is relatively more computationally expensive than other algorithms such as those in \cite{sparselab} but outperforms them all in terms of the accuracy of the identification as measured by the RNMSE. In future work, we plan to improve our proposed algorithm by exploiting further the structure of the optimisation problem at hand and reducing the associated algorithmic complexity. Another issue is that we assume that only process noise is present, and thus do not directly take into account measurement noise. We are currently working on an extension of the method allowing the incorporation of measurement noise into the presented framework.

\section{ACKNOWLEDGEMENT}
The authors gratefully acknowledge the support of Microsoft Research through the PhD Scholarship Program of Mr Wei Pan. Dr Ye Yuan and Dr Jorge Gon\c{c}alves acknowledge the support from EPSRC (project EP/I03210X/1). Dr Jorge Gon\c{c}alves also acknowledges the support from FNR and ERASysAPP. Dr Guy-Bart Stan gratefully acknowledges the support of the EPSRC Centre for Synthetic Biology and Innovation at Imperial College London (project EP/G036004/1) and the EPSRC Fellowship for Growth (project EP/M002187/1). The authors would like to thank Dr Wei Dai, Prof Mauricio Barahona and Dr Aivar Sootla (Imperial College London) for helpful discussions.

\clearpage
\onecolumn
\begin{center}
{\Large \textbf{\appendix}}
\end{center}
\numberwithin{equation}{subsection}

\subsection{Proof of Theorem \ref{theorem:hyper}}
\label{proof:1}
We first re-express $\mean$ and $\variance$ in (\ref{mean}) and (\ref{variance}) using the Woodbury inversion identity:
\begin{eqnarray}
\mean&=& \bGamma \dic^\top (\sbl)^{-1} \by=\noiseinv\variance\dic^\top\by, \label{mean2} \\
\variance&=& \bGamma- \bGamma \dic^\top (\sbl)^{-1} \dic \bGamma=(\bGamma^{-1}+\noiseinv\dic^\top \dic)^{-1}. \label{variance2}
\end{eqnarray}
Since the data likelihood $\Prob(\by|\bw)$ is Gaussian, we can write the integral for the marginal likelihood in (\ref{mass}), as
\begin{equation}
\begin{aligned}
&\int \mathcal{N}(\by|{\dic} {\bw},\noise\bI)\prod_{j=1}^{N}\bN(w_j|0,\hyper_j)\prior(\hyper_j)d\bw \\
=&\left(\frac{1}{2\pi\noise}\right)^{M/2}\left(\frac{1}{2\pi}\right)^{N/2}
\int \exp\{-E(\bw)\}d\bw \prod_{j=1}^{N}\frac{\prior(\hyper_j)}{\sqrt{\gamma_j}},
\label{integral}
\end{aligned}
\end{equation}
where $E(\bw)=\frac{1}{2\noise}\|\by-\dic\bw\|^2+\frac{1}{2}\bw^\top\bGamma^{-1}\bw, \ \bGamma=\diag(\bgamma)$. Equivalently, we get
\begin{eqnarray}
E(\bw)=\frac{1}{2}(\bw-\mean)^\top\variance^{-1}(\bw-\mean)+E(\by),
\label{integral2}
\end{eqnarray}
where $\mean$ and $\variance$ are given by (\ref{mean2}) and (\ref{variance2}). Using the Woodbury inversion identity, we obtain:
\begin{equation}
\begin{aligned}
E(\by) =\frac{1}{2}\left(\noiseinv\by^\top\by-\noiseinv\by^\top\dic\variance\variance^{-1}\variance\dic^\top\by\noiseinv\right) =\frac{1}{2} \by^\top(\sbl)^{-1}\by.
\label{data-dependent-term0}
\end{aligned}
\end{equation}
Using (\ref{integral2}), we can evaluate the integral in (\ref{integral}) to obtain
$$\int \exp\{-E(\bw)\}d\bw=\exp\{-E(\by)\}(2\pi)^{N/2}|\variance|^{1/2}.$$
Exploiting the determinant identity, we have $|\bGamma^{-1}||\sbl | =|\noise\bI| | \bGamma^{-1}+\noiseinv\dic^\top \dic |$, 
from which we can compute the first term in (\ref{hypercostfunction}) as
$\log \left\vert \sbl \right\vert =-\log |\variance|+M\log\noise+\log  |\bGamma|$.
Then applying a $-2\log(\cdot)$ transformation to (\ref{integral}), we have
\begin{equation}
\begin{aligned}
&-2\log \int \Prob(\by|\bw)\prod_{j=1}^{n}\bN(w_j|0,\hyper_j)\prior(\hyper_j)d\bw  \notag \\
= &-\log |\variance|+M\log2\pi \noise+\log  |\bGamma|+\by^\top(\sbl)^{-1}\by+\sum\nolimits_{j=1}^{N}p(\hyper_j) \notag \\
= &  \log \left\vert \sbl \right\vert+M\log 2\pi \noise +\by^\top(\sbl)^{-1}\by+\sum\nolimits_{j=1}^{N}p(\hyper_j).
\label{integral3}
\end{aligned}
\end{equation}
From (\ref{mass}), we then obtain $\hat{\bgamma} = \arg \min\limits_{\bgamma\geq\mathbf{0}} \{\log \left\vert \sbl\right\vert +\by^\top(\sbl)^{-1}\by+\sum\nolimits_{j=1}^{N}p(\hyper_j)\}.$ We compute the posterior mean to get an estimate of $\bw$: $\hat{\bw}=\E(\bw|\by;\hat{\bgamma})=\hat{\bGamma} \dic^\top (\noise\mathbf{I+{\dic}\hat{\bGamma}}{{\dic}}^\top)^{-1} \by$ where $\hat{\bGamma}=\diag[\hat{\bgamma}]$.

\subsection{Proof of Lemma \ref{lemma:concave}}
\label{proof:2}
We first show that the data-dependent term in (\ref{hypercostfunction}) is convex in $\bw$ and $\bgamma$. From (\ref{mean2}), (\ref{variance2}) and (\ref{data-dependent-term0}), the data-dependent term can be re-expressed as
\begin{equation}
\begin{aligned}
&\by^\top\left(\sbl\right)^{-1}\by  \\
=& \noiseinv\by^\top\by-\noiseinv\by^\top\dic\variance \dic^\top\noiseinv\by \\
=&  \noiseinv\|\by-\dic\mean\|_2^2+ \mean^\top\bGamma^{-1}\mean\\
=&\min_{\bx} \{\noiseinv\| \by-\dic\bx\|_{2}^{2}+\bx^\top\bGamma^{-1}\bx\},
\label{data-dependent-term}
\end{aligned}
\end{equation}
where $\mean$ is the posterior mean defined in (\ref{mean}). It can easily be shown that the minimisation problem is convex in $\bw$ and $\bgamma$, where  $\bGamma\define \diag[\bgamma]$.

Next we define $h(\bgamma)\define \log|\sbl|+\sum\nolimits_{j=1}^{N}p(\hyper_j) \label{hr}$, and show $h(\bgamma)$  is a concave function with respect to~$\bgamma$. $\log|\cdot|$ is concave in the space of positive semi-definite matrices. Moreover, $\sbl$ is an affine function of $\bgamma$ and is positive semidefinite for any $\bgamma\geq 0$. This implies that $\log \left\vert \sbl \right\vert$ is a concave, nondecreasing function of $\bgamma$.  Since we adopt a super-Gaussian prior with potential function $\prior(\hyper_j), \forall j,$ as described in Proposition \ref{super-Gaussian}, a direct consequence is that $p(\hyper_j)=-\log\prior(\hyper_j)$ is concave.

\subsection{Proof of Theorem \ref{theorem:MAP}}
\label{proof:3}
Using the data-dependent term in (\ref{data-dependent-term}), together with $\bL_{\bgamma}(\bgamma)$  in (\ref{hypercostfunction}), we can create a strict upper bounding auxiliary function on $\bL_{\bgamma}(\bgamma)$ as $\bL_{\bgamma,\bw}(\bgamma,\bw)= \noiseinv\| \by-\dic\bw\|_{2}^{2}+\bw^\top\bGamma^{-1}\bw+\log|\sbl|+\sum\nolimits_{j=1}^{N}p(\hyper_j)$. When we minimise over $\bgamma$ instead of $\bw$, we obtain
\begin{eqnarray}
\bL_{\bw}(\bw)
&\define& \min\limits_{\bgamma\geq\mathbf{0}}\bL_{\bgamma,\bw}(\bgamma,\bw) \notag \\
&=& \noiseinv\| \by-\dic\bw\|_{2}^{2}+\min\limits_{\bgamma\geq \mathbf{0}}\{\bw^\top\bGamma^{-1}\bw+\log|\sbl|+\sum\nolimits_{j=1}^{N}p(\hyper_j)\}.
\end{eqnarray}
Then for $\bw$ with convex constraints as described in Assumption \ref{assumption-constraints}, we obtain the formulation in Theorem~\ref{theorem:MAP}.

From the derivations in (\ref{data-dependent-term}), we can clearly see that the estimate of the stochastic variable $\bw$ is the poseterior mean $\mean$ defined in (\ref{mean}).

\subsection{Proof of Lemma \ref{lemma:gsb}}
\label{proof:4}

It is shown in Lemma \ref{lemma:concave} that $h(\bgamma)$ is concave with respect to $\bgamma\geq 0$. According to the duality lemma (see Sec. 4.2 in \cite{jordan1999introduction}), we can express the concave function $h(\bgamma)$ as
$
h(\bgamma)
=\min_{\bgamma^{*}\geq 0}
\left<\bgamma^{*}, \bgamma\right>-h^{*}(\bgamma^{*}),
$
where $h^{*}(\bgamma^{*})$ is defined as the concave conjugate of
$h(\bgamma)$ and is given by
$
h^{*}(\bgamma^{*})
=\min_{\bgamma\geq 0} \left<\bgamma^{*}, \bgamma\right>-h(\bgamma).
$

From the proof of Lemma \ref{lemma:concave}, the data-dependent term $\by^\top\left(\sbl\right)^{-1}\by$ can be re-expressed as
$\min_{\bw} \{\noiseinv\| \by-\dic\bw\|_{2}^{2}+\bw^\top\bGamma^{-1}\bw\}$.
Therefore we can create a strict upper bounding auxiliary function $\bL_{\bgamma,\bw}(\bgamma,\bw)$ on $\bL_{\bgamma}(\bgamma)$ in (\ref{hypercostfunction}) by considering the fact that, in the dual expression,
$
\bL_{\bgamma,\bw}(\bgamma,\bw)
\define \left<\bgamma^{*}, \bgamma\right>-h^{*}(\bgamma^{*})+\by^\top\left(\sbl\right)^{-1}\by
= \noiseinv\| \by-\dic\bw\|_{2}^{2}+\sum\nolimits_j\left({w_j^2}/{\hyper_j}+\hyper_j^{*}\hyper_j\right)-h^{*}(\bgamma^{*}).
$
We can then re-express $g_{\sbsb}(\bw)$ as
\begin{equation}
g_{\sbsb}(\bw)=\min\limits_{\bgamma,\bgamma{*}\geq \mathbf{0}}\left\{\sum\nolimits_j\left({w_j^2}/{\hyper_j}+\hyper_j^{*}\hyper_j\right)-h^{*}(\bgamma^{*})\right\}.
\label{gsbl-dual1}
\end{equation}
$g_{\sbsb}(\bw)$ is minimised over $\bgamma$ when
$
\hyper_j=|w_j|/\sqrt{\hyper^{*}_{j}},\ \forall j.
$
Substituting this expression into $g_{\sbsb}(\bw)$, we get
\begin{equation}
g_{\sbsb}(\bw)=\min\limits_{\bgamma^{*}\geq \mathbf{0}}\left\{\sum\nolimits_j2\sqrt{\hyper^{*}_{j}}|w_j|-h^{*}(\bgamma^{*})\right\}.
\label{gsbl-dual2}
\end{equation}
This indicates that $g_{\sbsb}(\bw)$ can be represented as a minimum over upper-bounding hyperplanes in $\|\bw\|_1$, and thus must be concave. $g_{\sbsb}(\bw)$ thus promotes sparsity. Moreover, $g_{\sbsb}(\bw)$ must be non-decreasing since $\bgamma^{*}\geq \mathbf{0}$.

\subsection{Proof of Theorem \ref{theorem:convergence}}
\label{proof:convergence}

The proof is in one-to-one correspondence with that of the Global Convergence Theorem \cite{zangwill1969nonlinear}.
\begin{enumerate}
\item The mapping $\bA(\cdot)$ is compact. Since any element of $\bgamma$ is bounded, ${\bL}\left( \bgamma\right)$ will not diverge to infinity.
In fact, for any fixed $\by$, $\dic$ and $\bgamma$, there will always exist a radius $r$ such that for any $\|\bgamma^{(0)}\|\leq 0$, $\|\bgamma^{(k)}\|\leq 0$.
\item We denote $\bgamma'$ as the non-minimising point of $\bL(\bgamma'')<\bL(\bgamma')$, $\forall$ $\bgamma'' \in \bA(\bgamma')$. At any non-minimising $\bgamma'$ the auxiliary objective function $\bL_{(\bgamma^{*})'}$ obtained from $\bgamma^{*}_{\sbsb}$ will be strictly tangent to $\bL(\bgamma)$ at $\bgamma'$. It will therefore necessarily have a minimum elsewhere since the slope at $\bgamma'$ is nonzero by definition. Moreover, because the $\log |\cdot|$ function is strictly concave, at this minimum the actual cost function will be reduced still further. Consequently, the proposed updates represent a valid descent function \cite{zangwill1969nonlinear}.
\item $\bA(\cdot)$ is closed at all non-stationary points.
\end{enumerate}


\newpage
\clearpage
\onecolumn

\clearpage

\begin{center}
	{\Large \textbf{Supplementary Material}}
\end{center}

\makeatletter
\renewcommand{\thesection}{S\arabic{section}}   
\renewcommand{\thesubsection}{S\arabic{subsection}}   
\renewcommand{\thetable}{S\arabic{table}}   
\renewcommand{\thefigure}{S\arabic{figure}}

\section{A Motivating Example}
\label{example:motivate}
In this example of a system with states and inputs memories, we will show how to construct the expanded dictionary matrix by adding candidate nonlinear functions.
\begin{example}
\label{example:NARX3}
As an illustrative example, we consider  the following model of polynomial terms for a single-input single-output (SISO) nonlinear autoregressive system with exogenous input (NARX model) \cite{leontaritis1985input}:
\begin{equation}
\begin{aligned}
x(t_{k+1})=0.7x^5(t_{k})x(t_{k-1})-0.5x(t_{k-2})
+0.6u^4(t_{k-2})-0.7x(t_{k-2})u^2(t_{k-1})
+\xi(t_k),
\label{NARX1}
\end{aligned}
\end{equation}
with $x, u, \xi \in {\mathbb{R}}$.
We can write \eqref{NARX1} in extended form as:
\begin{equation}
\begin{aligned}
x(t_{k+1})&=w_1+w_2x(t_k)+\ldots+w_{m_x+2}x(t_{k-m_x})+\ldots+w_{N}x^{d_x}(t_{k-m_x})u^{d_u}(t_{k-m_u})+\xi(t_k) \\
&= \bw^\top\bF(x(t_{k}),\ldots, x(t_{k-m_x}), u(t_{k}),\ldots,u(t_{k-m_u}))+\xi(t_{k}),
\label{NARX3}
\end{aligned}
\end{equation}
where $d_x$ (resp. $d_u$) is the degree of the output (resp. input); $m_x$ (resp. $m_u$) is the maximal memory order of the output (resp. input); $\bw^\top=[w_{1},\ldots,w_{N}]\in \R^{N}$ is the weight vector; and $\bF(\bx(t_{k}),\ldots, \bx(t_{k-m_\bx}), \bu(t_{k}),\ldots,\bu(t_{k-m_\bu}))=[f_{1}(\cdot),
\ldots,f_{N}(\cdot)]^\top\in \R^{N}$ is the dictionary functions vector.
By identification of \eqref{NARX3} with the NARX model (\ref{NARX1}), we can easily see that $d_{x}=5$, $d_{u}=4$,  $m_{x}=2$, $m_{u}=2$.
To define the dictionary matrix, we consider all possible monomials up to degree $d_{x}=5$ (resp. $d_u=4$) and up to memory order $m_x=5$ (resp. $m_u=2$) in $x$ (resp. $u$).
This yields $\bF(\cdot) \in \R^{1960}$ and thus $\bw \in \R^{1960}$. Since $\bv \in \R^{4}$, only 4 out of the 1960 associated weights $w_i$ are nonzero. 
\end{example}

\section{Illustrative Numerical Examples}\label{sec:examples}
\subsection{Experiment setup}

We hereafter present two classic examples that we use to illustrate our proposed method and on which we apply Algorithm~\ref{alg:weight}. 
To implement Algorithm~\ref{alg:weight}, we use CVX, a popular package for specifying and solving convex programs \cite{grant2008cvx}. The algorithm is implemented in MATLAB R2013a. The calculations were performed on a standard laptop computer (Intel Core i5 2.5GHz with 8GB RAM). In the examples that follow, we set the pruning threshold (mentioned in Remark~\ref{remark:threshold}) to $10^{-4}$, i.e., $\|w_j\|_2^2 / \|\bw\|_2^2<10^{-4}$.

\subsection{Two classic examples}
\begin{example}
\label{example:o}
In this example, we consider a classical dynamical system in systems/synthetic biology, the repressilator, which we use to illustrate the reconstruction problem at hand.
The repressilator is a synthetic three-gene regulatory network where the dynamics of mRNAs and proteins follow an oscillatory behaviour \cite{elowitz2000son}.
A discrete-time mathematical description of the repressilator, which includes both
transcription and translation dynamics, is given by the following set of discrete-time equations:
\begin{equation}
\begin{aligned}
x_{1}(t_{k+1}) &=x_{1}(t_{k})+(t_{k+1}-t_{k})\left[-\hyper _{1}x_{1}(t_{k})+\frac{\alpha _{1}}{(1+x_{6}^{n_{1}}(t_{k}))}\right]+\xi_1(t_{k}),  \notag \\
x_{2}(t_{k+1}) &=x_{2}(t_{k}) +(t_{k+1}-t_{k})\left[-\hyper_{2}x_{2}(t_{k})+\frac{\alpha _{2}}{(1+x_{4}^{n_{2}}(t_{k}))}\right]+\xi_2(t_{k}),  \notag \\
x_{3}(t_{k+1})&=x_{3}(t_{k})+(t_{k+1}-t_{k})\left[-\hyper _{3}x_{3}(t_{k})+\frac{\alpha _{3}}{(1+x_{5}^{n_{3}}(t_{k}))}\right]+\xi_3(t_{k}),
\notag \\
x_{4}(t_{k+1}) &=x_{4}(t_{k}) +(t_{k+1}-t_{k})\left[-\hyper _{4}x_{4}(t_{k})+\beta _{1}x_{1}(k)\right]+\xi_4(t_{k}),  \notag \\
x_{5}(t_{k+1})&=x_{5}(t_{k}) +(t_{k+1}-t_{k})\left[-\hyper _{5}x_{5}(k)+\beta _{2}x_{2}\right]+\xi_5(t_{k}),  \notag \\
x_{6}(t_{k+1}) &=x_{6}(t_{k}) +(t_{k+1}-t_{k})\left[-\hyper _{6}x_{6}(t_{k})+\beta _{3}x_{3}(t_{k})\right]+\xi_6(t_{k}).
\label{oscillator}
\end{aligned}
\end{equation}
Here, $x_{1},x_{2},x_{3}$ (resp. $x_{4},x_{5},x_{6}$) denote the concentrations of the mRNA transcripts (resp. proteins) of genes 1, 2, and 3, respectively.
$\xi_i$, $\forall i$ are i.i.d. Gaussian noise.
$\alpha _{1},\alpha
_{2},\alpha _{3}$ denote the maximum promoter strength for their corresponding gene,
$\hyper_{1},\hyper _{2},\hyper _{3}$ denote the mRNA degradation rates, $\hyper
_{4},\hyper _{5},\hyper _{6}$ denote the protein degradation rates, $\beta
_{1},\beta _{2},\beta _{3}$ denote the protein production rates, and
$n_{1},n _{2},n _{3}$ the Hill coefficients.
The set of equations in (\ref{oscillator}) corresponds to a topology where gene $1$ is repressed by gene $2$,
gene $2$ is repressed by gene $3$, and gene $3$ is repressed by gene $1$.
Take gene 1 for example. The hill coefficient $n_{1}$ will typically have a value within a range from $1$ to $4$ due to biochemical constraints. The core question here is:
how can we determine the topology and kinetic parameters of the set of equations in (\ref{oscillator}) from time series data of $x_{1}, \dots, x_{6}$?

Note that we do not assume \emph{a priori} knowledge of the form of the nonlinear functions appearing on the right-hand side of the equations in \eqref{oscillator}, e.g., whether the degradation obeys first-order or enzymatic catalysed dynamics or whether the proteins are repressors or activators. It should also be noted that many linear and nonlinear functions can be used to describe the dynamics of GRNs in terms of biochemical kinetic laws, e.g., first-order functions $f(\left[ S\right] )=\left[ S\right]$, mass action functions $f(\left[ S_{1}\right] ,\left[ S_{2}\right] )=\left[ S_{1}\right]\cdot \left[ S_{2}\right]$,  Michaelis-Menten functions $f(\left[ S\right])=V_{\max }\left[ S\right] /(K_{M}+\left[ S\right] )$, or Hill functions $f(\left[ S\right] )=V_{\max }\left[ S\right] ^{n}/(K_{M}^{n}+[S]^{n})$.
These kinetic laws typical of biochemistry and GRN models will aid in the definition of the dictionary function matrix.
Next we show how the network construction problem of the repressilator model in \eqref{oscillator} can be formulated in a linear regression form.

Following the procedure described in Section II 
of the main text, we construct a candidate dictionary matrix $\dic$, by selecting as candidate basis functions, nonlinear functions typically used to represent terms appearing in biochemical kinetic laws of GRN models.
As a proof of concept, we only consider Hill functions as potential nonlinear candidate functions. The set of Hill functions with Hill coefficient $h$, both in activating and repressing from, for each of the $6$ state variables are:
\begin{equation}
\begin{aligned}
\text{hill}_h(t_{k}) &\define &\left[ \frac{1}{1+x_{1}^{h}(t_{k})},\ldots ,
\frac{1}{1+x_{6}^{h}(t_{k})}, \frac{x_{1}^{h}(t_{k})}{1+x_{1}^{h}(t_{k})},\ldots ,\frac{x_{6}^{h}(t_{k})}{1+x_{6}^{h}(t_{k})}\right]_{1\times 12}, 
\end{aligned}
\end{equation}
where $h$ represents the Hill coefficient. In what follows we consider
that the Hill coefficient can take any of the following integer values: $1$, $2$, $3$ or $4$. Since there are 6 state variables, we can construct the dictionary matrix $\dic$ with $6$ (dictionary functions for linear terms) $+(4*12)$ (dictionary functions for Hill functions) $=54$ columns.
\begin{equation}
\begin{aligned}
\dic =\left[
\begin{array}{ccccccc}
x_1(t_0) & \ldots & x_6(t_0) & \text{hill}_1(t_0) & \ldots &
\text{hill}_4(t_0)  \\
\vdots &  & \vdots & \vdots &  & \vdots \\
x_1(t_{M-1}) & \ldots & x_6(t_{M-1}) & \text{hill}_1(t_{M-1}) & \ldots &
\text{hill}_4(t_{M-1})
\end{array}
\right] \in \R^{M\times (6+48)}.
\label{oscillator-dic}
\end{aligned}
\end{equation}
Then the output can be defined as
$$\by_i\define
\left[\frac{x_i(t_{1})-x_i(t_{0})}{t_{1}-t_0},\ldots,\frac{x_i(t_{M})-x_i(t_{M-1})}{t_{M}-t_{M-1}}\right]^\top\in {\mathbb{R}}^{M\times 1}, i=1,\ldots,6.$$
Considering the dictionary matrix $\dic$ given in (\ref{oscillator-dic}), the corresponding target $\bw_i$ for the ``correct'' model in (\ref{oscillator}) should be:
\begin{equation}
\begin{aligned}
\bw_{true} &= [\bw_1, \bw_2,\bw_3,\bw_4,\bw_5, \bw_6]  \\
&=
\left[
\begin{array}{cccccc}
-\gamma_{1}(=-0.3) & 0 & 0 & \beta _{1}(=1.4) & 0 & 0 \\
0 & -\gamma_{2}(=-0.4) & 0 & 0 & \beta _{2}(=1.5) & 0 \\
0 & 0 & -\gamma_{3}(=-0.5)  & 0 & 0 & \beta _{3}(=1.6) \\
0 & 0 & 0 & -\gamma_{4} (=-0.2) & 0 & 0 \\
0 & 0 & 0 & 0 & -\gamma_{5} (=-0.4) & 0 \\
0 & 0 & 0 & 0 & 0 & -\gamma_{6}(=-0.6)  \\
\0_{47\times 1} & \0_{45\times 1} & \0_{46\times 1} &  &  &  \\
\alpha _{1}(=4)  & \alpha _{2}(=3)  & \alpha _{3}(=5)  & \0_{48\times 1} & \0_{48\times 1} &
\0_{48\times 1} \\
\0_{0\times 1}  & \0_{2\times 1} & \0_{1\times 1} &  &  &  
\end{array}
\right] . 
\label{truepara}
\end{aligned}
\end{equation}
with values in brackets indicating the correct parameter values.

To generate the time-series data, we took `measurements' every $t_{k+1}-t_{k}=1$ between $t = 0$ and $t = 50$ (arbitrary units) from random initial conditions  which are drawn from a standard uniform distribution on the open interval $(0,1)$.
Thus a total of 51 measurements for each state are collected (including the initial value). It should be noted that the number of rows is less than the number of columns in the dictionary matrix.

\end{example}

\begin{example}
\label{example:k}
A classical example in physics, engineering and biology is the Kuramoto oscillator network \cite{strogatz2000kuramoto}. We consider a network where the Kuramoto oscillators are nonidentical (each has its own natural oscillation frequency $\omega_i$) and the coupling strengths between nodes are not the same. The corresponding discrete-time dynamics can be described by
\begin{equation}
\begin{aligned}
{\phi_i} (t_{k+1})={\phi_i} (t_{k})+(t_{k+1}-t_{k})\left[\omega_i+\sum\limits_{j=1, j\neq i}^{n} w_{ij}g_{ij}(\phi_j(t_k)-\phi_i(t_k))+\xi_i(t_k)\right], i=1, \ldots, n,
\label{Kuramoto}
\end{aligned}
\end{equation}
where $\phi_i\in[0,2\pi)$ is the phase of oscillator $i$, $\omega_i$ is its natural frequency, and the coupling function $g_{ij}$ is a continuous and smooth function, usually taken as $\sin, \ \forall i, j$. $w_{ij}$ represent the coupling strength between oscillators $i$ and $j$ thus $[w_{ij}]_{n\times n}$ defines the topology of the network.
Here, assuming we don't know the exact form of $g_{ij}$, we reconstruct from time-series data of the individual phases $\phi_i$ a dynamical network consisting of $n$ Kuramoto oscillators, i.e., we identify the coupling functions $g_{ij}(\cdot)$ as well as the model parameters, i.e., $\omega_i$ and $w_{ij}$, $i, j=1,\dots,n$.

To define the dictionary matrix $\dic$, we assume that all the dictionary functions are functions of a pair of  state variables only and consider $5$ candidate coupling functions $g_{ij}$:  $\sin(x_j-x_i)$, $\cos(x_j-x_i)$, $x_j-x_i$, $\sin^2 (x_j-x_i)$, and $\cos^2 (x_j-x_i)$. Based on this, we define the dictionary matrix as
\begin{equation}
\begin{aligned}
\dic_{ij}(x_j(t_k),x_i(t_k)) \define& [\sin(x_j(t_k)-x_i(t_k)), \cos(x_j(t_k)-x_i(t_k)), x_j(t_k)-x_i(t_k), \notag \\
& \sin^2 (x_j(t_k)-x_i(t_k)), \cos^2 (x_j(t_k)-x_i(t_k))] \in \R^{5}. \notag
\end{aligned}
\end{equation}

To also take into account the natural frequencies, we add to the last column of $\dic_i$ a unit vector. This leads to the following dictionary matrix $\dic_i$:
\begin{equation}
\begin{aligned}
\dic_{i} &\define
\left[
\begin{array}{cccc}
\dic_{i1}(x_1(t_0),x_i(t_0))& \ldots  &\dic_{in}(x_n(t_0),x_i(t_0)) & 1\\
\vdots  & \vdots  & \vdots  & \vdots \\
\dic_{i1}(x_1(t_{M-1}),x_i(t_{M-1})) & \ldots  & \dic_{in}(x_n(t_{M-1}),x_i(t_{M-1})) & 1
\end{array}
\right]
\in {\mathbb{R}}^{M\times {(5n+1)}}. \notag
\end{aligned}
\end{equation}
Then the output can be defined as
$$\by_i\define
\left[\frac{\phi_i(t_{1})-\phi_i(t_{0})}{t_{1}-t_0},\ldots,\frac{\phi_i(t_{M})-\phi_i(t_{M-1})}{t_{M}-t_{M-1}}\right]^\top\in {\mathbb{R}}^{M\times 1}, i=1,\ldots,n.$$

To generate the time-series data, we simulated a Kuramoto network with $n=100$ oscillators, for which $10\%$ of the non-diagonal entries of the weight matrix $[w_{ij}]_{n\times n}$ are nonzero (assuming $g_{ii}$ and $w_{ii}$ are zeros), and the non-zero $w_{ij}$ values are drawn from a standard uniform distribution on the interval $[-10,10]$. The natural frequencies $\omega_{i}$ are drawn from a normal distribution with mean $0$ and variance $10$.
In order to create simulated data, we simulated the discrete-time model \eqref{Kuramoto} and took `measurements data points' every $t_{k+1}-t_{k}=0.1$ between $t = 0$ and $t = 45$ (in arbitrary units) from random initial conditions drawn from a standard uniform distribution on the open interval $(0,2\pi)$.
Thus a total of 451 measurements for each oscillator phase $\phi_i \in \R^{451\times 501}$ are collected (including the initial value).
Once again, it should be noted that the the number of rows of the dictionary matrix is less than that of columns.

\end{example}

\subsection{Algorithmic Performance Comparisons in terms of Signal-to-Noise Ratio}
\numberwithin{equation}{subsection}
We here investigate the performance of various algorithms including ours (Algorithm~\ref{alg:weight} in the main text)
for different signal-to-noise ratios of the data generated for Example \ref{example:o} and Example \ref{example:k}. We define the signal-to-noise ratio (SNR) as $\mathrm{SNR(dB)}\define 20 \log_{10} (\|\dic \bw_{\mathrm{true}}\|_2/\|\bXi\|_2)$. We considered SNRs ranging from 0 dB to 25 dB for each generated weight. To compare the reconstruction accuracy of the various algorithms considered, we use the root of normalised mean square error (RNMSE) as a performance index, i.e., $\|\hat{\bw}-\bw\|_2/\|\bw\|_2$, where $\hat{\bw}$ is the estimate of the true weight $\bw$.
For each SNR, we performed 200 independent experiments and calculated the average RNMSE for each SNR over these 200 experiments.
In each ``experiment'' of Example \ref{example:o}, we simulated  the repressilator model with random initial conditions drawn from a standard uniform distribution on the open interval $(0,1)$. The parameters were drawn from a standard uniform distribution with the true values $\bw_{true}$ in (\ref{truepara}) taken as the mean and variations around the mean values no more than $10\%$ of the true values. In MATLAB, one can use \texttt{$\bw_{true}$.*(0.9 + 0.2*rand(54,6))} to generate the corresponding parameter matrix for each experiment.
In each ``experiment'' of Example \ref{example:k}, we simulated a Kuramoto network with $n=100$ oscillators, for which $10\%$ of the non-diagonal entries of the weight matrix $[w_{ij}]_{n\times n}$ were nonzero (assuming $g_{ii}$ and $w_{ii}$ are always zero). The non-zero $w_{ij}$ values were drawn from a standard uniform distribution on the interval $[-10,10]$. The natural frequencies $\omega_{i}$ were drawn from a normal distribution with mean $0$ and variance $10$.

Based on these settings, we compared Algorithm~\ref{alg:weight} with nine other state-of-the-art sparse linear regression algorithms available at \cite{sparselab}. \cite{sparselab} provides access to a free MATLAB software package managed by David Donoho and his team and contains various tools for finding sparse solutions of linear systems, least-squares with sparsity, various pursuit algorithms, and more. In Table I, we briefly describe the algorithms provided in \cite{sparselab} and used for this comparison.  In Fig. \ref{fig:nmseO} and Fig. \ref{fig:nmseK}, we plot, for various SNRs, the average RNMSE  obtained using our algorithm and other algorithms in \cite{sparselab} for the problems considered in Example \ref{example:o} and Example \ref{example:k} respectively.
In Fig. \ref{fig:timeO} and Fig. \ref{fig:timeK}, we plot, for the various SNRs considered, the average computational running time required by our algorithm and the other algorithms from \cite{sparselab} for the problems considered in Example \ref{example:o} and Example \ref{example:k} respectively. During this comparison, the inputs for the algorithms listed in Table I are always the same, i.e., the dictionary matrix $\dic$ and the data contained in $\by$. The initialisation and pre-specified parameters for these algorithms were set to their default values provided in \cite{sparselab}. Interested readers can download the package from \cite{sparselab} and reproduce the results presented here under the default settings of the solvers therein.

It should be noted that the dictionary matrices in all the experiments are rank deficient, i.e., neither column rank nor row rank are full. As a consequence, both the MP and OMP algorithm fail to converge or yield results with extremely large RNMSE. As these two algorithms cannot satisfactorily be used, they have been removed from the comparison results presented in Figures~\ref{fig:nmseO} to~\ref{fig:timeK}.

\begin{center}
\begin{table}
\begin{tabular}{ | p{2.5cm} |  p{3.5cm} | p{9.5cm} |}
\hline
Abbreviation & Solver name in \cite{sparselab} & Method  \\ \hline
BP         &     \verb"SolveBP.m"          &   Basis Pursuit \\ \hline
IRWLS        &     \verb"SolveIRWLS.m"      &    Iteratively ReWeighted Least Squares\\ \hline
ISTBlock       &     \verb"SolveISTBlock.m"       &   Iterative Soft Thresholding, block variant with least squares projection\\ \hline
LARS      &     \verb"SolveLasso.m"         &   Implements the LARS algorithm\\ \hline
MP      &     \verb"SolveMP.m"             &    Matching Pursuit\\ \hline
OMP      &     \verb"SolveOMP.m"            &   Orthogonal Matching Pursuit\\ \hline
PFP      &     \verb"SolvePFP.m"            &   Polytope Faces Pursuit algorithm\\ \hline
Stepwise      &     \verb"SolveStepwise.m"       &    Forward Stepwise\\ \hline
StOMP       &     \verb"SolveStOMP.m"          &   Stagewise Orthogonal Matching Pursuit  \\ \hline
\end{tabular}
\caption{Description of algorithms  in \cite{sparselab} used for comparisons}
\end{table}
\end{center}

\begin{figure}[H]
\center
\includegraphics[scale=0.4]{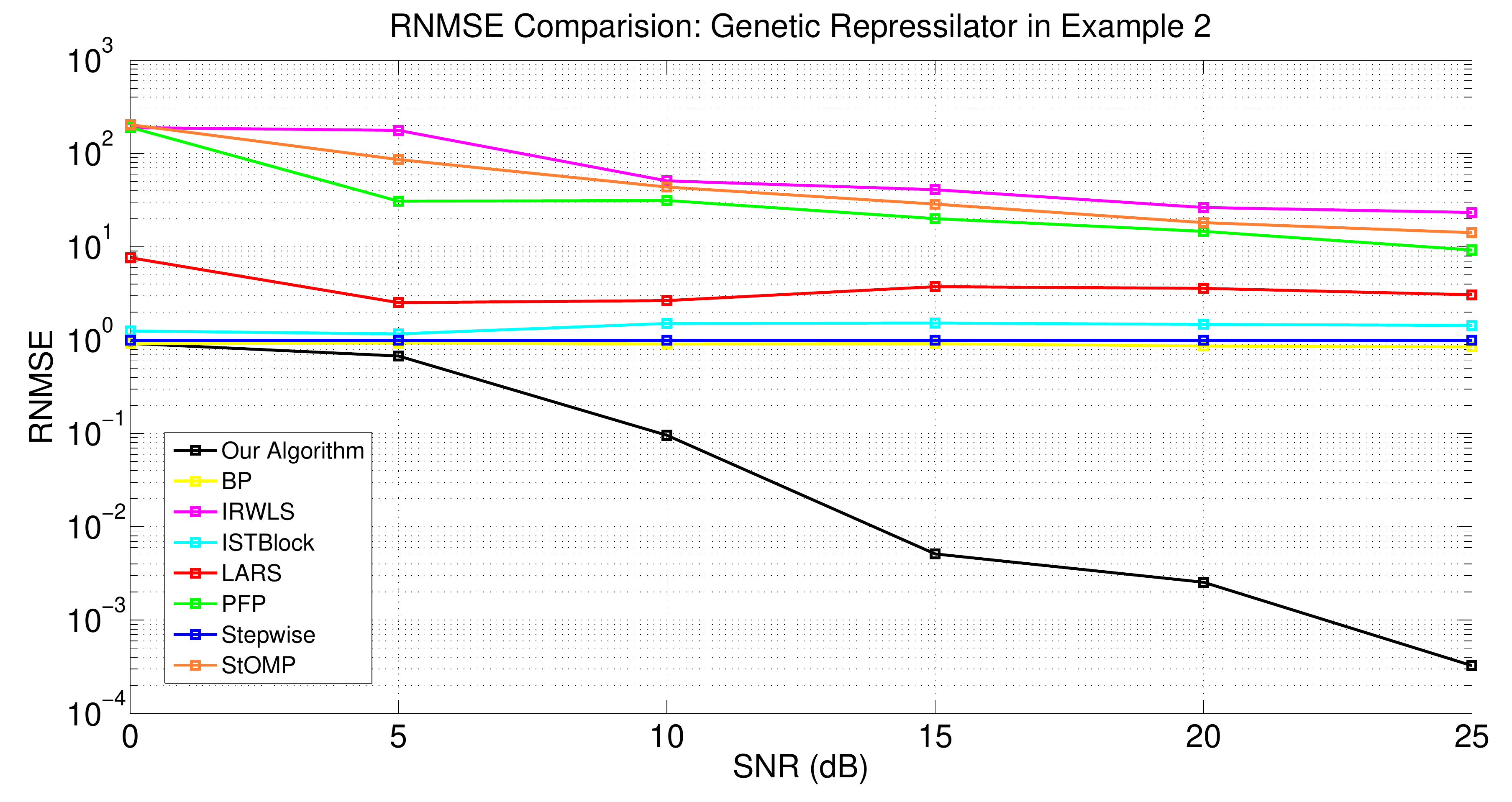}
\caption{Root of Normalised Mean Square Error averaged over 200 independent experiments for the signal-to-noise ratios 0 dB, 5 dB, 10 dB, 15 dB, 20 dB, and 25 dB in Example \ref{example:o}.
}
\label{fig:nmseO}
\end{figure}

\begin{figure}[H]
\center
\includegraphics[scale=0.4]{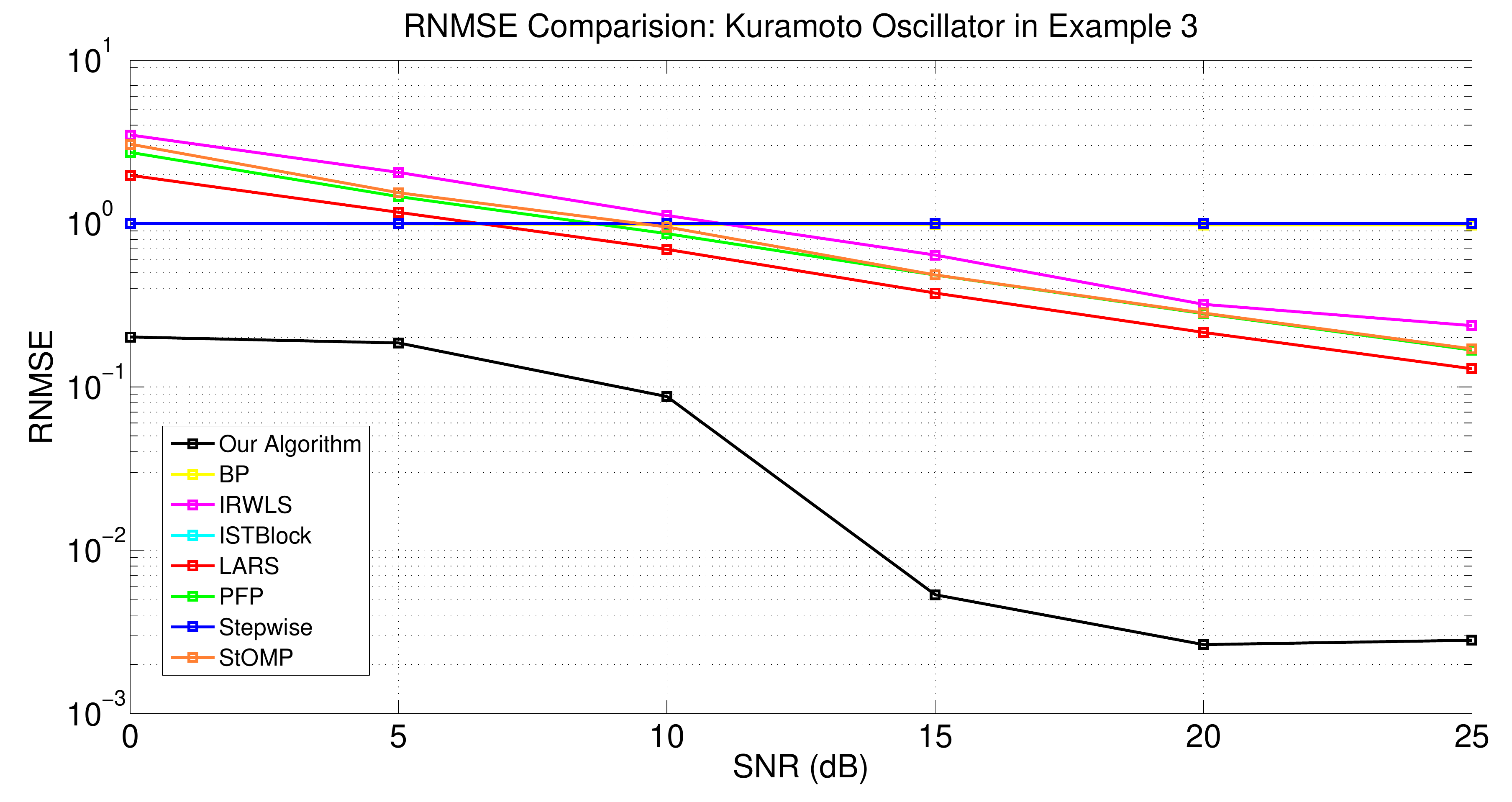}
\caption{Root of Normalised Mean Square Error averaged over 200 independent experiments for the signal-to-noise ratios 0 dB, 5 dB, 10 dB, 15 dB, 20 dB, and 25 dB in Example \ref{example:k}.
	}
\label{fig:nmseK}
\end{figure}
 
\begin{figure}[H]
\center
\includegraphics[scale=0.4]{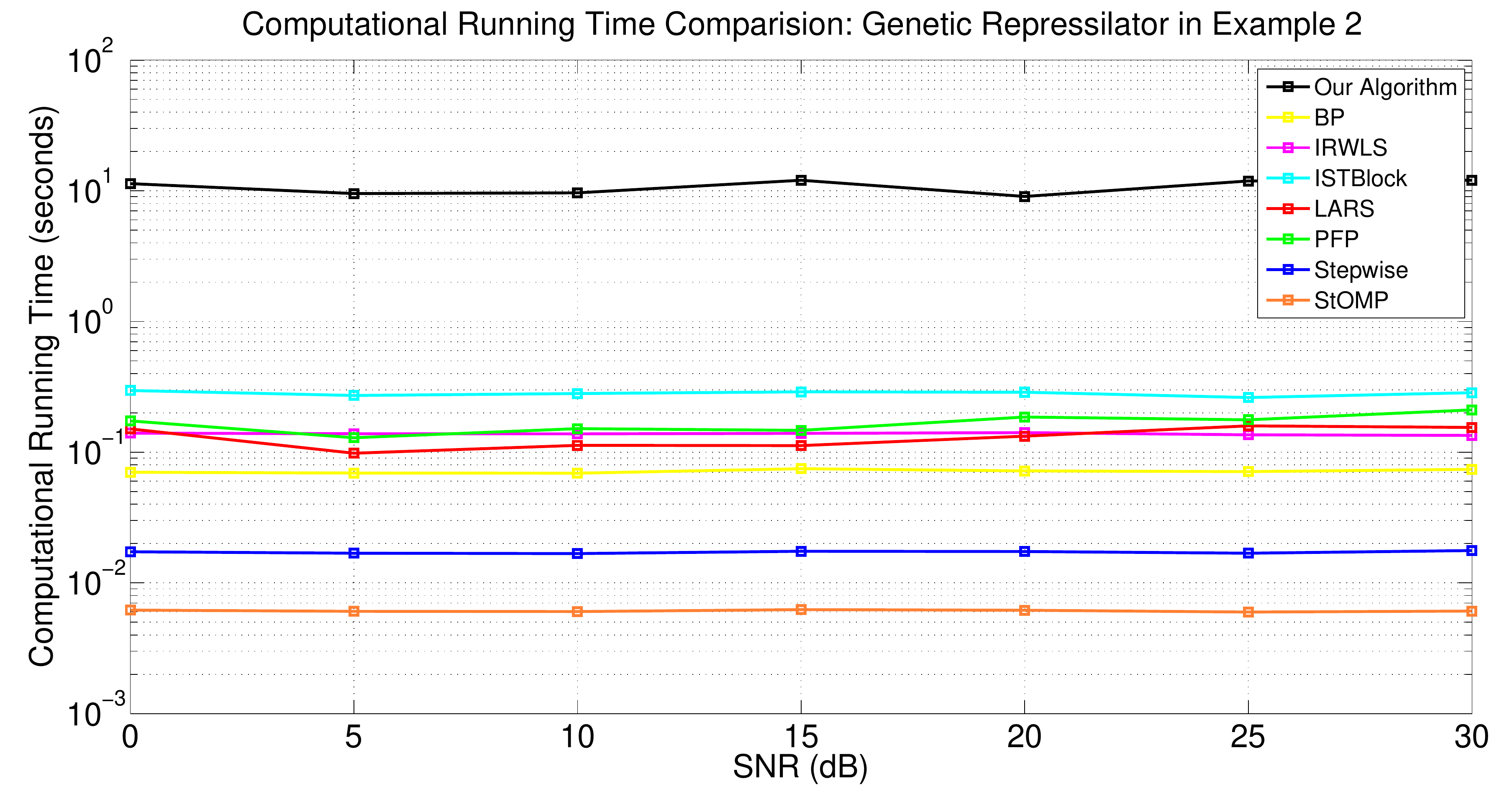} 
\caption{Computational running time averaged over 200 independent experiments for the signal-to-noise ratios 0 dB, 5 dB, 10 dB, 15 dB, 20 dB, and 25 dB in Example \ref{example:o}.}
\label{fig:timeO}
\end{figure}

\begin{figure}[H]
\center
\includegraphics[scale=0.4]{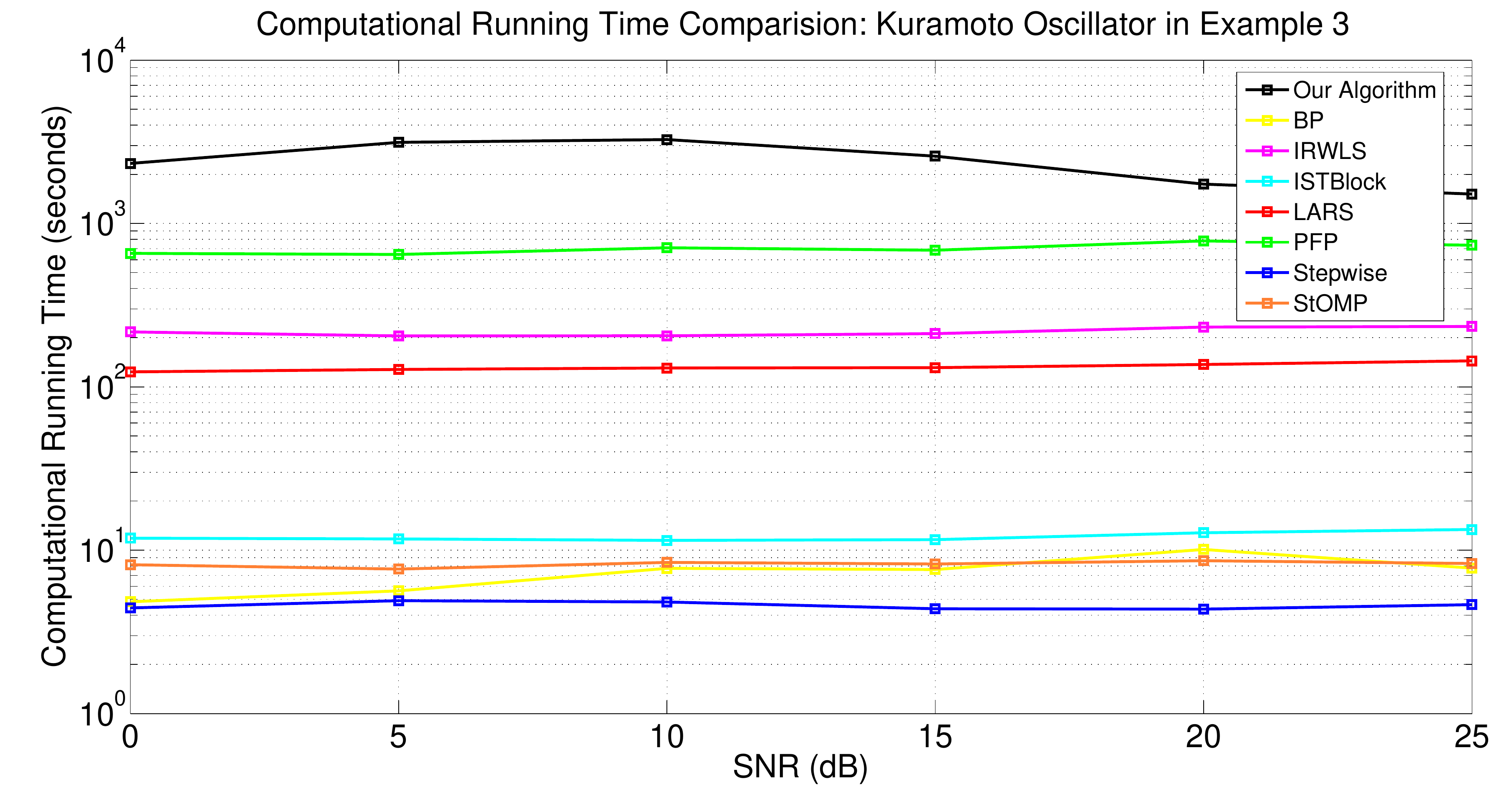} 
\caption{Computational running time averaged over 200 independent experiments for the signal-to-noise ratios 0 dB, 5 dB, 10 dB, 15 dB, 20 dB, and 25 dB in Example \ref{example:k}.}
\label{fig:timeK}
\end{figure}

\subsection{Discussion}
It can be seen from Fig. \ref{fig:nmseO} and Fig. \ref{fig:nmseK} that our algorithm outperforms all the other algorithms in \cite{sparselab} in terms of RNMSE.
However, our algorithm requires more computational running time compared to the other algorithms. There are potentially two reasons for this. The first one is that our algorithm is implemented using the CVX package as a parser \cite{grant2008cvx}. Parsers similar to CVX include YALMIP \cite{lofberg2004yalmip}. CVX and YALMIP call generic SDP solvers, e.g., SDPT3 \cite{toh1999sdpt3} or SeDuMi \cite{sturm1999using}, to solve the convex optimisation problem at hand (we use SeDuMi). While these solvers are reliable for wide classes of optimisation problems, they are not specifically optimised in terms of algorithmic complexity to exploit the specific structure of particular problems, such as ours. The second reason comes from the $5^{th}$ step of Algorithm \ref{alg:weight} where the  matrix $\noise\mathbf{I}+\dic \mathbf{U}^{(k)} \mathbf{W}^{(k)} \dic^\top \in \R^{M \times N}$ has to be inverted to update the weights for the next iteration. Though a pruning rule has been discussed in Remark \ref{remark:threshold}, such inversion at each iteration is inevitable compared to the algorithms considered in \cite{sparselab}. 
It should also be noted from Fig. \ref{fig:nmseO} and Fig. \ref{fig:nmseK} that the RNMSE of some algorithms is independent of the SNR. This may be due to the fact that the coherence of the dictionary matrix is close to $1$ regardless of the SNR. In such case, these algorithms cannot identify the model correctly even when the SNR is high. Therefore, these algorithms are not appropriate for nonlinear system identification.
In future work, we plan to improve our proposed algorithm by exploiting further the structure of the optimisation problem at hand and reducing the associated algorithmic complexity.


\begin{thebibliography}{10}

\bibitem{tacappendix}
\emph{Appendix and supplementary material}, \emph{arXiv:1408.3549}.

\bibitem{ljung1999system}
L.~Ljung, \emph{System Identification: Theory for the User}.\hskip 1em plus
  0.5em minus 0.4em\relax Prentice Hall, 1999.

\bibitem{billings2013nonlinear}
S.~A. Billings, \emph{Nonlinear system identification: NARMAX methods in the
  time, frequency, and spatio-temporal domains}.\hskip 1em plus 0.5em minus
  0.4em\relax John Wiley \& Sons, 2013.

\bibitem{sjoberg1995nonlinear}
J.~Sj{\"o}berg, Q.~Zhang, L.~Ljung, A.~Benveniste, B.~Delyon, P.~Glorennec,
  H.~Hjalmarsson, and A.~Juditsky, ``Nonlinear black-box modeling in system
  identification: a unified overview,'' \emph{Automatica}, vol.~31, no.~12, pp.
  1691--1724, 1995.

\bibitem{haber1990structure}
R.~Haber and H.~Unbehauen, ``Structure identification of nonlinear dynamic
  systems: survey on input/output approaches,'' \emph{Automatica}, vol.~26,
  no.~4, pp. 651--677, 1990.

\bibitem{barahona1996detection}
M.~Barahona and C.~Poon, ``Detection of nonlinear dynamics in short, noisy time
  series,'' \emph{Nature}, vol. 381, no. 6579, pp. 215--217, 1996.

\bibitem{wiener1966nonlinear}
N.~Wiener, ``Nonlinear problems in random theory,'' \emph{Nonlinear Problems in
  Random Theory, by Norbert Wiener, pp. 142. ISBN 0-262-73012-X. Cambridge,
  Massachusetts, USA: The MIT Press, August 1966.}, vol.~1, 1966.

\bibitem{narendra1990identification}
K.~Narendra and K.~Parthasarathy, ``Identification and control of dynamical
  systems using neural networks,'' \emph{Neural Networks, IEEE Transactions
  on}, vol.~1, no.~1, pp. 4--27, 1990.

\bibitem{leontaritis1985input}
I.~Leontaritis and S.~Billings, ``Input-output parametric models for non-linear
  systems part i: deterministic nonlinear systems,'' \emph{International
  journal of control}, vol.~41, no.~2, pp. 303--328, 1985.

\bibitem{bai1998optimal}
E.~Bai, ``An optimal two-stage identification algorithm for hammerstein--wiener
  nonlinear systems,'' \emph{Automatica}, vol.~34, no.~3, pp. 333--338, 1998.

\bibitem{kollar2009probabilistic}
D.~Kollar and N.~Friedman, \emph{Probabilistic graphical models: principles and
  techniques}.\hskip 1em plus 0.5em minus 0.4em\relax The MIT Press, 2009.

\bibitem{pearl1988probabilistic}
J.~Pearl, \emph{Probabilistic reasoning in intelligent systems: networks of
  plausible inference}.\hskip 1em plus 0.5em minus 0.4em\relax Morgan Kaufmann,
  1988.

\bibitem{spirtes2000causation}
P.~Spirtes, C.~Glymour, and R.~Scheines, \emph{Causation, prediction, and
  search}.\hskip 1em plus 0.5em minus 0.4em\relax The MIT Press, 2000, vol.~81.

\bibitem{barber2010graphical}
D.~Barber and A.~Cemgil, ``Graphical models for time-series,'' \emph{Signal
  Processing Magazine, IEEE}, vol.~27, no.~6, pp. 18--28, 2010.

\bibitem{bach2004learning}
F.~R. Bach and M.~I. Jordan, ``Learning graphical models for stationary time
  series,'' \emph{Signal Processing, IEEE Transactions on}, vol.~52, no.~8, pp.
  2189--2199, 2004.

\bibitem{materassi2012problem}
D.~Materassi and M.~V. Salapaka, ``On the problem of reconstructing an unknown
  topology via locality properties of the wiener filter,'' \emph{Automatic
  Control, IEEE Transactions on}, vol.~57, no.~7, pp. 1765--1777, 2012.

\bibitem{pan2012cdc}
W.~Pan, Y.~Yuan, J.~Gon{\c{c}}alves, and G.-B. Stan, ``Reconstruction of
  arbitrary biochemical reaction networks : A compressive sensing approach,''
  in \emph{IEEE 51st Annual Conference on Decision and Control (CDC)}, 2012.

\bibitem{Candes2005decoding}
E.~Cand{\`e}s and T.~Tao, ``Decoding by linear programming,'' \emph{Information
  Theory, IEEE Transactions on}, vol.~51, no.~12, pp. 4203--4215, 2005.

\bibitem{donoho2006compressed}
D.~Donoho, ``Compressed sensing,'' \emph{Information Theory, IEEE Transactions
  on}, vol.~52, no.~4, pp. 1289--1306, 2006.

\bibitem{mairal2008discriminative}
J.~Mairal, F.~Bach, J.~Ponce, G.~Sapiro, and A.~Zisserman, ``Discriminative
  learned dictionaries for local image analysis,'' pp. 1--8, 2008.

\bibitem{mairal2010online}
J.~Mairal, F.~Bach, J.~Ponce, and G.~Sapiro, ``Online learning for matrix
  factorization and sparse coding,'' \emph{The Journal of Machine Learning
  Research}, vol.~11, pp. 19--60, 2010.

\bibitem{TianshiTAC}
T.~Chen, M.~Andersen, L.~Ljung, A.~Chiuso, and G.~Pillonetto, ``System
  identification via sparse multiple kernel-based regularization using
  sequential convex optimization techniques,'' \emph{Automatic Control, IEEE
  Transactions on}, vol.~59, no.~11, pp. 2933--2945, 2014.

\bibitem{cerone2011enforcing}
V.~Cerone, D.~Piga, and D.~Regruto, ``Enforcing stability constraints in
  set-membership identification of linear dynamic systems,'' \emph{Automatica},
  vol.~47, no.~11, pp. 2488--2494, 2011.

\bibitem{zavlanos2011inferring}
M.~Zavlanos, A.~Julius, S.~Boyd, and G.~Pappas, ``Inferring stable genetic
  networks from steady-state data,'' \emph{Automatica}, vol.~47, no.~6, pp.
  1113--1122, 2011.

\bibitem{tibshirani1996regression}
R.~Tibshirani, ``Regression shrinkage and selection via the lasso,''
  \emph{Journal of the Royal Statistical Society. Series B (Methodological)},
  pp. 267--288, 1996.

\bibitem{dai2009subspace}
W.~Dai and O.~Milenkovic, ``Subspace pursuit for compressive sensing signal
  reconstruction,'' \emph{Information Theory, IEEE Transactions on}, vol.~55,
  no.~5, pp. 2230--2249, 2009.

\bibitem{donoho2003optimally}
D.~Donoho and M.~Elad, ``Optimally sparse representation in general
  (nonorthogonal) dictionaries via $\ell_1$ minimization,'' \emph{Proceedings of the
  National Academy of Sciences}, vol. 100, no.~5, pp. 2197--2202, 2003.

\bibitem{Candes2006robust}
E.~Cand{\`e}s, J.~Romberg, and T.~Tao, ``Robust uncertainty principles: Exact
  signal reconstruction from highly incomplete frequency information,''
  \emph{Information Theory, IEEE Transactions on}, vol.~52, no.~2, pp.
  489--509, 2006.

\bibitem{bishop2006pattern}
C.~Bishop, \emph{Pattern Recognition and Machine Learning}.\hskip 1em plus
  0.5em minus 0.4em\relax Springer New York, 2006, vol.~4.

\bibitem{palmer2006variational}
J.~Palmer, D.~Wipf, K.~Kreutz-Delgado, and B.~Rao, ``Variational {EM}
  algorithms for non-{Gaussian} latent variable models,'' \emph{Advances in
  neural information processing systems}, vol.~18, p. 1059, 2006.

\bibitem{wainwright2008graphical}
M.~Wainwright and M.~Jordan, ``Graphical models, exponential families, and
  variational inference,'' \emph{Foundations and Trends in Machine Learning},
  vol.~1, no. 1-2, pp. 1--305, 2008.

\bibitem{tipping2001sparse}
M.~Tipping, ``Sparse bayesian learning and the relevance vector machine,''
  \emph{The Journal of Machine Learning Research}, vol.~1, pp. 211--244, 2001.

\bibitem{boyd1987linear}
S.~Boyd, L.~El~Ghaoul, E.~Feron, and V.~Balakrishnan, \emph{Linear matrix
  inequalities in system and control theory}.\hskip 1em plus 0.5em minus
  0.4em\relax Society for Industrial Mathematics, 1987, vol.~15.

\bibitem{horn1990matrix}
R.~Horn and C.~Johnson, \emph{Matrix analysis}.\hskip 1em plus 0.5em minus
  0.4em\relax Cambridge university press, 1990.

\bibitem{jordan1999introduction}
M.~Jordan, Z.~Ghahramani, T.~Jaakkola, and L.~Saul, ``An introduction to
  variational methods for graphical models,'' \emph{Machine learning}, vol.~37,
  no.~2, pp. 183--233, 1999.

\bibitem{candes2008enhancing}
E.~Cand{\`e}s, M.~Wakin, and S.~Boyd, ``Enhancing sparsity by reweighted
  $\ell_1$ minimisation,'' \emph{Journal of Fourier Analysis and Applications},
  vol.~14, no.~5, pp. 877--905, 2008.

\bibitem{wipf2010iterative}
D.~Wipf and S.~Nagarajan, ``Iterative reweighted $\ell_1$ and $\ell_2$ methods
  for finding sparse solutions,'' \emph{IEEE Journal of Selected Topics in
  Signal Processing}, vol.~4, no.~2, pp. 317--329, 2010.

\bibitem{zangwill1969nonlinear}
W.~I. Zangwill, \emph{Nonlinear programming: a unified approach}.\hskip 1em
  plus 0.5em minus 0.4em\relax Prentice-Hall Englewood Cliffs, NJ, 1969.

\bibitem{grant2008cvx}
M.~Grant, S.~Boyd, and Y.~Ye, ``{CVX}: {{MATLAB}} software for disciplined
  convex programming,'' \emph{Online accessiable: http://cvxr.com}, 2008.

\bibitem{sparselab}
D.~L. Donoho, V.~C. Stodden, and Y.~Tsaig, ``About {SparseLab},'' \emph{Online
  accessible: http://sparselab.stanford.edu}, 2007.

\bibitem{yuan2011robust}
Y.~Yuan, G.~Stan, S.~Warnick, and J.~Goncalves, ``Robust dynamical network
  structure reconstruction,'' \emph{Special Issue on System Biology,
  Automatica}, vol.~47, pp. 1230--1235, 2011.

\bibitem{yuan2012decentralised}
Y.~Yuan, ``Decentralised network prediction and reconstruction algorithms,''
  Ph.D. dissertation, 2012.

\end{thebibliography}

\begin{thebibliography}{1}


\bibitem{leontaritis1985input}
I.~Leontaritis and S.~Billings, ``Input-output parametric models for non-linear
systems part i: deterministic nonlinear systems,'' \emph{International
journal of control}, vol.~41, no.~2, pp. 303--328, 1985.

\bibitem{grant2008cvx}
M.~Grant, S.~Boyd, and Y.~Ye, ``{CVX}: {MATLAB} software for disciplined convex
programming,'' \emph{Online accessible: http://cvxr.com/}, 2008.

\bibitem{elowitz2000son}
M.~Elowitz and S.~Leibler, ``{A synthetic oscillatory network of
transcriptional regulators},'' \emph{Nature}, vol. 403, no. 6767, pp. 335--8,
2000.

\bibitem{strogatz2000kuramoto}
S.~Strogatz, ``From kuramoto to crawford: exploring the onset of
synchronisation in populations of coupled oscillators,'' \emph{Physica D:
Nonlinear Phenomena}, vol. 143, no.~1, pp. 1--20, 2000.

\bibitem{sparselab}
D.~L. Donoho, V.~C. Stodden, and Y.~Tsaig, ``About {SparseLab},'' \emph{Online
accessible: http://sparselab.stanford.edu/}, 2007.

\bibitem{lofberg2004yalmip}
J.~Lofberg, ``{YALMIP}: A toolbox for modeling and optimization in MATLAB,'' in
\emph{Computer Aided Control Systems Design, 2004 IEEE International
Symposium on}, pp.
284--289, 2004.

\bibitem{toh1999sdpt3}
K.-C. Toh, M.~J. Todd, and R.~H. T{\"u}t{\"u}nc{\"u}, ``{SDPT3}--a {MATLAB}
software package for semidefinite programming, version 1.3,''
\emph{Optimization Methods and Software}, vol.~11, no. 1-4, pp. 545--581,
1999.

\bibitem{sturm1999using}
J.~F. Sturm, ``Using {SeDuMi} 1.02, a {MATLAB} toolbox for optimization over
symmetric cones,'' \emph{Optimization methods and software}, vol.~11, no.
1-4, pp. 625--653, 1999.

\end{thebibliography}

\end{document}